\documentclass[a4paper,reqno]{amsart}

\pdfoutput=1

\usepackage[utf8]{inputenc}
\usepackage[T1]{fontenc}
\usepackage{lmodern}
\usepackage{amsthm, amssymb, amsmath, amsfonts, mathrsfs,latexsym}
\usepackage[mathscr]{euscript}
\usepackage{cancel}
  \usepackage{graphics,tikz}
\usepackage{microtype}
\usetikzlibrary{decorations.pathreplacing}

\usepackage[pagebackref,colorlinks=true,pdfpagemode=none,urlcolor=blue,
linkcolor=blue,citecolor=blue]{hyperref}
\usepackage{caption}
\usepackage{hyperref}

\usepackage{relsize}

\usepackage{amsmath,amsfonts,amssymb,amsthm}
\usepackage{amsthm, amssymb, amsmath, amsfonts, mathrsfs}

\usepackage{mathrsfs,fge}
\usepackage{MnSymbol}
\usepackage{scalerel} % for scaling the cube

\usepackage{xcolor}
\usepackage{accents}

\usepackage[normalem]{ulem}
\usepackage{bbm}

\usepackage{cleveref}

\definecolor{labelkey}{gray}{.8}
\definecolor{refkey}{gray}{.8}

\definecolor{darkred}{rgb}{0.9,0.1,0.1}
\definecolor{darkgreen}{rgb}{0,0.5,0}

\numberwithin{equation}{section}

\newcommand{\BB}{{\cal B}}

\newcommand{\FF}{{\cal F}}

\newcommand{\BR}{{\mathbb R}}

\newcommand{\si}{\sigma}
\newcommand{\om}{\omega}

\newcommand{\E}{{\mathbb E}}
\newcommand{\bbE}{{\mathbb E}}
\newcommand{\R}{{\mathbb R}}
\newcommand{\bbR}{{\mathbb R}}
\newcommand{\al}{\alpha}
\newcommand{\la}{\lambda}
\newcommand{\bn}{\bf n}

\newcommand{\essinf}{\mathop{\mathrm{ess\,inf}}}
\newcommand{\esssup}{\mathop{\mathrm{ess\,sup}}}

\newcommand{\cal}{\mathcal}

\newtheorem{theorem}{\bf Theorem}[section]
\newtheorem{proposition}[theorem]{\bf Proposition}
\newtheorem{lemma}[theorem]{\bf Lemma}%[subsection]
\newtheorem{corollary}[theorem]{\bf Corollary}

\theoremstyle{definition}
\newtheorem{definition}[theorem]{Definition}

\newtheorem{remark}[theorem]{Remark}

\numberwithin{equation}{section}

\begin{document}

\title[Hopf type lemmas]{Hopf type lemmas for subsolutions 
  of integro-differential equations. Supplement: discussion on the
  applications of the results}

\author{Tomasz Klimsiak  and  Tomasz Komorowski}

\address[Tomasz Klimsiak]{Institute of Mathematics, Polish Academy Of Sciences,
ul. \'{S}niadeckich 8,   00-656 Warsaw, Poland, \and 
\newline
Faculty of
Mathematics and Computer Science, Nicolaus Copernicus University,
Chopina 12/18, 87-100 Toruń, Poland, e-mail: {\tt tomas@mat.umk.pl}}

\address[Tomasz Komorowski]{Institute of Mathematics, Polish Academy
  Of Sciences, ul. \'Sniadeckich 8, 00-636 Warsaw, Poland, e-mail: {\tt komorow@hektor.umcs.lublin.pl}}

\date{}
\maketitle
\begin{abstract}
 In the paper we prove a lower bound  for
  subsolutions of the integro-differential equation:   $-Au+cu=0$ in a domain $D$.
It states that there exists a Borel function
  $\psi$,  strictly positive on $D$,  depending   only on the
 coefficients of the operator $A$,  $c$ and $D$
such that for any subsolution $u(\cdot)$, that satisfies $\sup_{y\in
  D_S}u(y)\ge0$,  one can find a constant
$a>0$ (that in general depends on $u$), for which
$\sup_{y\in D_S}u(y)-u(x)\ge a\psi(x)$, $x\in
D$. 
The bound  is valid for 
a wide class of L\'evy type integro-differential operators $A$, 
non-negative, bounded and measurable function $c$ and  a quite general domain $D\subset \BR^d$.
Here $D_S$ is a certain set containing the closure of $D$ and  determined by the support of the Levy jump measure associated with
$A$.  The main assumptions made about $A$ are that: there exists a strong Markov
solution to the martingale problem  associated with  the operator and its resolvent satisfies some
minorization condition.
  This type of a result we
 call  the {\em generalized Hopf lemma}.

For certain classes of operators  
the constant $a$  could be taken to be equal to $ \sup_{y\in
  D_S}u(y)$. We refer to such a result as a {\em quantitative version} of the
 Hopf lemma.  In some cases     a non-negative   eigenfunction
corresponding to the operator in $D$ can be admitted as the function $\psi$ appearing in 
the lower bound.  In particular, this occurs when the transition
probability semigroup associated with $A$ is    ultracontractive.

\end{abstract}

% \footnotetext{{\em Mathematics Subject Classification:}
% Primary  35R06; Secondary 35R05, 45K05, 47G20, 35D99}

% \footnotetext{{\em Keywords:} Renormalized solution, Dirichlet
% form,  measure data, Markov semigroup, Markov process, Green
% function }

%\footnotetext{This work was supported by Polish National Science Centre
%(grant no. 2016/23/B/ST1/01543).}

\section{Introduction}
\label{sec1}
Let $m$ be a $\sigma$-finite Borel measure, let $D\subset \bbR^d$ be a
domain, i.e.  an open and connected set,
and let  $A$ be 
an  integro-differential operator of the form
 \begin{equation}
\label{A}
\begin{split}
&Au(x)=\frac12{\rm Tr}({\bf{ Q}}(x)\nabla^2u(x))+b(x)\cdot\nabla u(x)
\\&\qquad \qquad+\int_{\BR^d}\left(u(x+y)-u(x)-\frac{y\cdot \nabla
    u(x)}{1+|y|^2}\right)N(x,dy),\,\quad x\in D
\end{split}
\end{equation}
for any $u\in C^2(D)\cap C_b(\BR^d)$. Consider the following condition.

\vspace*{0.25cm}
\begin{center}
\begin{minipage}[c][2,75cm][t]{0,85\textwidth}
\textbf{Minorization condition:} there exist $\al\ge0$, a Borel measurable function
$\psi_D^\al:D\to[0,+\infty)$ and $\nu^\alpha_D$ - a
$\si$-finite  Borel measure  on $D$ - such that 
\begin{equation*}
%\label{RDa.int}
R^D_\alpha f(x)\ge \psi_D^\al(x) \int_D f d\nu_D^\alpha,\quad x\in
D,\, f\in B^+_b(D).
\end{equation*}
Given particular $\al$, $\psi_D^\al$ and  $\nu^\alpha_D$ we shall
refer to the above hypothesis as condition $M(\alpha,\psi^\alpha_D,\nu^\alpha_D)$.
\end{minipage}
\end{center}
Here $R^D_\alpha$ is the resolvent   of the operator $A$ on $D$,
see \eqref{RDt} below for the precise definition, and $B^+_b(D)$ is
the space of all non-negative, bounded, Borel measurable functions.

The main result of the present paper, see Theorem  \ref{th3.hl.nuc1},  states that 
if the minorization condition $M(\alpha,\psi^\alpha_D,\nu^\alpha_D)$ holds 
with    strictly positive $\psi_D^\alpha$, the   measure
$m$   is absolutely continuous with respect to $\nu_D^\alpha$ and some additional,
  rather natural,
assumptions about   $A$ (stated below) are satisfied, then 
there exists a strictly positive function $\psi_D$ on $D$ such that 
for any {\em   subsolution} $u$ to
\begin{equation} 
\label{eq3.1ab}
(-A+c)v(x)=0,\quad x\in D
\end{equation} 
satisfying $\bar u_{D_S}:=\sup_{D_S}u\ge 0$ there exists $a>0$ for which
\begin{equation}
\label{basic-bound}
\bar u_{D_S}-u(x)\ge a\psi_D(x),\quad x\in D.
\end{equation}
Moreover, if $\alpha\ge \sup_D c$, then we may take $\psi_D=\psi^\alpha_D$.
The set $D_S$ appearing above is the extended closure of $D$. In case
$A$ is local it is the usual closure of $D$.  For non-local
operators it contains the closure of the domain and  is determined by the L\'{e}vy
jump measure,  see \eqref{exbexc} below for the precise definition.

{Concerning the hypotheses on the operator $A$, we suppose that:
\begin{itemize}
\item[A1)] the entries of the symmetric matrix valued function ${\bf
 Q}(x)=[q_{i,j}(x)]_{i,j=1}^{d}$ and coordinates of the vector valued
$b(x)=(b_1(x),\ldots,b_d(x))$, $x\in \bbR^d$ are bounded and
Borel measurable.   In addition, 
${\bf
 Q}(x) $ is non-negative definite, i.e.
\begin{equation*}
%\label{la-K1}
\sum_{i,j=1}^dq_{i,j}(x)\xi_i\xi_j\ge 0,\quad x \in\BR^d,\, \xi=(\xi_1,\ldots,\xi_d)\in\BR^d,
\end{equation*}
%\end{itemize}
\item[A2)]  the {\em
  (L\'evy)
kernel} $N(x,dy)$ is a $\si$-finite Borel measure on
  $\R^d\setminus\{0\}$ for each $x\in \bbR^d$ and 
%for every $B\in B(\BR^d\setminus\{0\})$ the function
\begin{equation*}
%\label{Nx1}
N_*:=\sup_{x\in \bbR^d}\int_{\bbR^d} \min\{1,|y|^2\} N(x,dy)<+\infty.
\end{equation*}
\end{itemize}
% Such operators are sometimes called of the  Waldenfels type,
% see e.g. \cite[Chapter 10]{Taira2}. 
We assume also that:
\begin{itemize}
\item[A3)] 
$c$ is non-negative, bounded and Borel measurable on $D$.
\end{itemize}}
Besides the  
hypotheses   listed above, we shall also assume 
the existence of  a strong Markov solution of the martingale problem
associated with $A$, see hypothesis A4) formulated in Section \ref{SMM}.    

 When $N$ is non-trivial, then $A$ is non-local, i.e. the
 evaluation of $Au(x)$
no longer depends only on the values of $u$ in an arbitrarily small neighborhood of $x$.
To describe this dependence define the {\em range of non-locality} of $A$ at $x$ as:
\[
\mathcal S_x:=\big[z\in\bbR^d:\, z=x+y,\,\mbox{where } y\in \mbox{supp}\, N(x,\cdot)\big].
\]
Here $\mbox{supp}\, N(x,\cdot)$ denotes the support  of the measure,
i.e. the smallest closed set whose complement is of null measure. The
value of  $Au(x)$ depends therefore on the values of $u$ in an arbitrarily small
neighborhood of $x$ and  on its values   in $\mathcal S_x$.
Furthermore, for an open set $D\subset\BR^d$ we define  the range of
  non-locality of $A$ over $D$ as
\begin{equation}
\label{cSD}
\mathcal S(D)=\bigcup_{x\in D}\mathcal S_x.
\end{equation}
For example  if $N(x,dy)\equiv 0$, i.e. $A$ is  local, then $\mathcal
S(D)=\emptyset$. 
On the other hand, in the case of the fractional power of the free
Laplacian    $A=\Delta^{s/2}:=-(-\Delta)^{s/2}$, with
$s\in(0,2)$, see \eqref{Sal} below, we have $\mathcal S(D)=\BR^d$.

It is well known that in the case of  a  local operator $A$ and
a sufficiently regular $D$    the
solution of \eqref{eq3.1ab} (recall that $c\ge0$)  is uniquely
determined by its values on the boundary of the domain $\partial D$.
However, this,  in general, is  no longer  true for non-local operators, %as then bounded solutions  of \eqref{eq3.1ab} are  determined  by their values on the  complement of $D$, 
see e.g. \cite{dynkin}.
For this
reason, we introduce the {\em extended boundary}
and {\em extended closure} 
%Thus, we add to the boundary $\partial D$  an exterior set $\mathcal S(D)\setminus D$.
\begin{equation}
\label{exbexc}
\partial_{S}D:=\partial D\cup(\mathcal S(D)\setminus D),\quad D_S:=\partial_{S}D\cup D,
\end{equation}
respectively. The above notions appear  naturally  when we want to formulate a counterpart of the weak maximum principle for operators of the form \eqref{A}. Namely, if $u$ is
a  subsolution of \eqref{eq3.1ab} in $D$, then 
 \begin{equation}
\label{032804-20}
\bar u_D\le   \bar u^+_{\partial_{S}D}
\end{equation}
(see Proposition \ref{prop012904-20}). Here and in what follows, for a given subset $B\subset \BR^d$, we denote $\bar u_B:=
\sup_{y\in B}u(y)$ and $u^+:=\max\{u,0\}$.

% The supremum  of subsolutions to \eqref{eq3.1ab}, which we are
% interested in the paper, is taken over ${\rm cl}\,D\cup\mathcal
% 

Concerning the   subsolution to \eqref{eq3.1ab}, we consider a  quite
general notion of  the   {\em
  weak  subsolution}, see Definition \ref{df3.1g} below. It is
probabilistically motivated and
 includes, besides the {\em classical} subsolutions also {\em   viscosity} and  {\em Sobolev
  subsolutions}, see Remark \ref{svc.w1}. We denote by $\mathcal U_c$
the class of all weak subsolutions
  to \eqref{eq3.1ab} and by $\mathcal U^+_c$ the set of $u\in\mathcal
  U_c$  that satisfy $\bar u_{D_S}\ge 0$.

In the
local  case, i.e.  when   $N\equiv0$, uniform ellipticity condition
holds on compact sets and the boundary is sufficiently
regular \eqref{basic-bound} is valid with  
$$
\psi_D(x)=\delta_D(x):= {\rm dist}(x,\partial D)
$$ - the distance of $x$ from
$\partial D$, see e.g. \cite[Section 2.3]{weinberger}. This is the
contents of the   Hopf lemma valid for second order ellitpic
operators, see also  \cite{GS} for the result for  a
fractional laplacian. 
 An estimate of the form \eqref{basic-bound} can  be therefore considered 
as   a
generalization  of the classical Hopf lemma
and we shall refer to $\psi_D$ appearing there as a {\em bottom function}.  
In the present paper, we take a comprehensive look at  the validity of \eqref{basic-bound} from a   probabilistic viewpoint and 
propose a unified approach to the problem for a wide class of
integro-differential operators.

The main result of the paper,   sketched in the foregoing, is
proved in Section \ref{NS-hopf}, see Theorem \ref{th3.hl.nuc1}. Some
complementary   results, concerning the relationships between the
generalized Hopf lemma, irreducibility, minorization condition and the strong maximum principle,  are shown in Section \ref{mis.1}. The
diagram presented below illustrates the relations     between them. The symbol ``$\psi\succ 0$ in $D$''
means $\psi(x)>0$ for all $x\in D$.
\begin{figure}[ht]

\begin{center}
\begin{tikzpicture}[scale=0.8]
\draw[thick]  (-8.7,0) rectangle (-4.2,2);
\draw[thick] (-0.7,0) rectangle (3.7,2);
\draw[thick]  (-8.7,-5) rectangle (-4.2,-3);
\draw[thick] (-0.7,-5) rectangle (3.7,-3);
\draw[thick] (-0.7,1) -- (3.7,1);
\node [below] at (-6.5,2.6){{\small Irreducibility}};
\node [below] at (-6.5,2){\mbox{\tiny{ For any $f\in B_b^+(D)$ }}};
\node [below] at (-6.5,1.6){ \mbox{\tiny{such that $\int_Df\,dm>0$ we have}}};
\node [below] at (-6.5,1){\tiny{$R^D_1f\succ 0$\,\,\, in $D.$}};
\node [below] at (-6.5,-3.2){\mbox{\tiny{ If $u\in\mathcal U_c^+$ and  there exists }}};
\node [below] at (-6.5,-3.7){\mbox{\tiny{ $x_0\in D$ such that $\bar u_{D_S}=u(x_0),$}}};
\node [below] at (-6.5,-4.2){\mbox{\tiny{ then $u\equiv \bar u_{D_S}\,\,\, m$-a.e.}}};
\node [below] at (-6.5,-5){{\small Strong Maximum Principle}};
\node [above] at (1.5,2){{\small Minorization condition}};
\node [below] at (1.5,2){\mbox{\tiny{ $\forall{\alpha\ge 0}\,\, M(\alpha,\psi_\alpha,\nu_\alpha)$ holds}}};
\node [below] at (1.5,1.5){\mbox{\tiny{ {\rm with}\,\, $\psi_\alpha\succ0\,\,\mbox{and}\,\, m\ll\nu_\alpha.$}}};
\node [below] at (1.5,1){\mbox{\tiny{ $\exists{\alpha\ge 0}\,\, M(\alpha,\psi_\alpha,\nu_\alpha)$ holds}}};
\node [below] at (1.5,0.5){\mbox{\tiny{ {\rm with}\,\, $\psi_\alpha\succ0\,\,\mbox{and}\,\,  m\ll\nu_\alpha.$}}};
\node [above] at (1.5,-5.6) {\mbox{\small{The Hopf lemma}}};
\node [below] at (1.5,-3) {\mbox{\tiny{ $\exists\,\, \psi_D\succ0$ such that  for any}}};
\node [below] at (1.5,-3.5) {\mbox{\tiny{  $u\in\mathcal U^+_c$ $\exists\,\, a>0$ for which}}};
\node [below] at (1.5,-4) {\mbox{\tiny{ $\bar u_{D_S}-u(x)\ge a\psi_D(x),\,\, x\in D.$}}};
\draw[->] (-4,1.7) -- (-1,1.7) node [above, text centered, midway]{\small{Theorem \ref{th3.hl.nuc1a}}}
node [below,midway]{{\small ($m$ is excessive)}};
\draw[->] (-1,0.3) -- (-4,0.3) node [above, midway]{\small{Remark \ref{rem.apx3}}};
\draw[->] (1.5,-0.1) -- (1.5,-2.9) node [above, rotate=90, text centered, midway]{{\small Theorem \ref{th3.hl.nuc1}}};
\draw[->] (-0.8,-2.9) -- (-4.2,-0.1) node[above, rotate=-40,midway]{{\small Theorem \ref{prop.apx1}}} node[below, rotate=-40,midway]{{\small ($m$ is excessive)}};
\draw[->] (-8,-2.9) -- (-8,-0.1) node [above, rotate=90, text centered, midway]{{\small Theorem \ref{prop012904-20v2}}} node [below, rotate=90, text centered, midway]{{\small  ($m$\,\,{\rm is excessive})}};
\draw[<-] (-5,-2.9) -- (-5,-0.1) node [above, rotate=90, text centered, midway]{{\small Theorem \ref{prop012904-20v2}}} ;
\end{tikzpicture}
\caption{} 
\label{fig1}
\end{center}
\end{figure}

In general, the constant $a$ appearing in \eqref{basic-bound} may depend on a
subsolution   in some implicit and complicated way and the  
bottom function $\psi_D$ is not given explicitly.
In Sections \ref{disc2}--\ref{sec3-2810-20} we formulate several results 
that yield additional information about  these objects.
From Proposition \ref{prop5.1} it follows that if the transition
probability semigorup $(P^D_t)$, coresponding to the Markov solution
of the martingale problem, is intrinsically  ultracontractive, then 
for any $\alpha>0$, the minorization condition $M(\alpha,\psi^\alpha_D,\nu^\alpha_D)$ holds with
\begin{equation}
\label{iul.apx1}
\psi^\alpha_D=\varphi_D,\quad \nu^\alpha_D(dx)=c_\alpha\hat \varphi_D(x)\,m(dx),
\end{equation}
where $c_\alpha>0$ is some constant, depending on $\alpha$, and $\varphi_D$ and $\hat\varphi_D$  are the principal
eigenfunctions for the semigroup  and its dual,  
respectively. Therefore, in particular, the Hopf lemma, as formulated
in \eqref{basic-bound}, is valid in this case with $\psi_D=\varphi_D$.
In fact, as it  can be seen from Theorem \ref{thm010211-20}, the
aforementioned bound holds
if and only if the $\varphi_D$-Doob transform of the canonical process
  has a uniformly
ergodic resolvent.
In Theorem \ref{thm010807-20}, we formulate sufficient conditions under which the
constant $a$ appearing in \eqref{basic-bound}  can be written in the form
\begin{equation}
\label{aa}
a=a'\bar u_{D_S}+\int_D(Au-cu)\,d\nu_{D}^\alpha,
\end{equation}
for some constant $a'>0$ independent of $u$. In particular, this, combined with \eqref{iul.apx1} shows that if $(P^D_t)$ is intrinsically  ultracontractive,
then there exists $a'>0$  such that
\begin{equation*}
%\label{012804-20zapx45v1}
\bar u_{D_S}-u(x)\ge a' \varphi_D(x)\left(\bar u_{D_S}+\int_D\hat\varphi_D(y)(A-c)u(y)\,dy\right)
\end{equation*}
for any subsolution $u$ to \eqref{eq3.1ab} satisfying  $\bar u_{D_S}\ge0$.
A particular example when \eqref{basic-bound} holds, 
with $\psi_D(x)=\varphi_D(x)=\delta_D^{s/2}(x)$ and the constant  $a$ given by \eqref{aa}, is  
furnished by the 
fractional laplacian 
  $\Delta^{s/2}$, in case  $D$ is bounded and of $C^{1,1}$
class, see Remark \ref{rm5.4}.

%\textcolor{red}{\bf DOTAD}

% Finally, we show  that the  $\varphi_D$-Hopf and
% its quantitative   counterpart hold, if the transition
% semigroup  
% associated with $A$ and the domain $D$, via the strong Markovian solution
% of the martingale problem,  is intrinsically ultracontractive, see Corollary
% \ref{cor5.1}. 

%and we focus in the paper on the behavior of subsolutions to \eqref{eq3.1ab} around such points.
Concerning the existing literature,   our  results are related to the
current research dealing with the  boundary regularity of solutions to
integro-differential equations, see
e.g. \cite{BL,Bogdan,BFV,CS1,CRS,GS,RS,RS1} and references
therein. Most of the  existing results
deal with the fractional Laplacian $\Delta^{s/2}$, i.e. the operator
of the form \eqref{A} with $\mathbf Q\equiv 0,\ b\equiv 0$ and
\begin{equation}
\label{eq.fjm}
N(x,dy)=c |y|^{-d-s}\,dy,\quad x,y\in\BR^d
\end{equation}
for some $s\in (0,2)$ and  $c>0$.
Although the equations with the fractional Laplacian are fundamental
and  their analysis is important in understanding the nature of non-local equations,
this class of operators is not sufficient for many applications. One
should keep in mind that even a  small  modification of the non-local
part of an operator, especially when its local part is degenerate, may
profoundly change the  regularity properties of solutions.
For example, it is  known  that  for positive bounded and continuous
$\Delta^{s/2}$-subharmonic  functions, with  $s\in(0,2)$,  on bounded smooth domains $D$ - i.e. satisfying  $\Delta^{s/2} u(x)\le 0,$ $x\in D$ -
the difference $\bar u_{\BR^d}-u(x)$ behaves near the boundary  like
$\delta_D^{s/2}(x)$, if $\bar u_{\BR^d}=u(\hat x)$ for some $\hat
x\in\partial D$. However,   
if  an innocent looking   (as one might be tempted to think) modification of $\Delta^{s/2}$
is made
by replacing constant $c $ in \eqref{eq.fjm} by a function $d(y)$
which is bounded from below and   above by  positive constants,
then   the differences $\bar u_{\BR^d}-u(x)$   for   the subharmonic functions
corresponding to the respective operators, in general,
are  not comparable with each other  near the boundary, see
\cite[Section 2.3]{RS1}. As we have already mentioned,  in
the present paper   we  strive for estimates of
the form \eqref{basic-bound} that are not related to some special features (such as
e.g. scaling properties of the L\'evy kernel) of the non-local
operator under the consideration.

The maximum principles both weak and strong, together with
the Hopf lemma in the case when the local part of $A$ satisfies
uniform ellipticity condition on compact sets can be found   e.g. in  \cite[Section I.4]{BCP}, \cite[Section 10.2]{Taira2} and
\cite[Appendix C]{Taira6}. This is also the topic of our paper
\cite{kk01} and an interested reader can find some additional
references therein.
Finally, we mention also that there exists a substantial literature
concerning the Hopf lemma and maximum principles  on
non-local partial differential equations with $\mathcal S(D)\subset
{\rm cl}\,D$ (then $D_S={\rm cl}\,D$ and $\partial_S D=\partial D$ ), see e.g. \cite[Appendix C]{Taira2} and the references therein.

%In the rest of the introduction we assume that $u(\hat x)=[u]_{{\rm cl}\,D\cup \mathcal S(D)}\ge 0$ for some $\hat x\in \partial D$.

Concerning the organization of the present paper, in Section
\ref{sec2} we
recall
%  introduce  the  notion of a weak subsolution (supersolution and solution) of
% \eqref{eq3.1ab}. To formulate it, we assume 
the notion of a strong Markov solution of the martingale problem
associated with the operator $A$, see the hypothesis A4) formulated in
Section \ref{SMM}.  Its existence (uniqueness is not required) is
assumed throughout the paper.
 % It is given by a family of path measures $(P_x)_{x\in\bbR^d}$ on the
% Skorokhod space ${\cal D}$, such that the
% functionals \eqref{Mtf} are martingales, see Definition \ref{df020311-20} for the
% precise formulation. The canonical process
% $X_t(\om):=\om(t)$, $t\ge0$,  $\om\in {\cal D}$ has the strong Markov
% property under each measure $P_x$, $x\in\bbR^d$. 
% We can   define the transition semigroup $(P_t^D)_{t\ge0}$ and  resolvent family
% $(R^D_{\al})_{\al\ge 0}$ corresponding to the Markov solution of the martingale
% problem, see \eqref{PDt} and \eqref{RDt}, respectively.
Under fairly mild assumptions on the coefficients of $A$, see Theorem \ref{thm2.3} (which
summarizes some of the existing results on the subject) one can guarantee the existence of
 such a   solution. 
In addition, throughout the   paper
 we shall suppose  that the canonical process
exits $D$ a.s. starting from any point of the domain, see condition ET) expressed by formula
\eqref{012603-20}. {Some useful conditions guaranteeing   ET) are presented in Section \ref{rm3.14}.}
 {This hypothesis holds e.g. if we assume 
that $D$ is bounded and 
the differential part of the operator is uniformly elliptic in $D$ (the
latter however is not assumed  in the present paper).} Finally, using   a strong Markov solution of the martingale
problem  we define  in Section \ref{sec2.5} the notion of a weak
subsolution (supersolution and solution) of
\eqref{eq3.1ab}. 

In Sections \ref{NS-hopf} and \ref{mis.1} we formulate and prove the main results of the
  paper.  In Sections \ref{disc2} and \ref{disc} various relationships
between  estimate \eqref{basic-bound}, the ultracontractivity  and
some   ergodic properties of   the
canonical process   are discussed.  
 Section \ref{sec3-2810-20} is devoted to the
formulation and the proof of the quantitative version of the Hopf
lemma. Finally,
in Section \ref{sec6zz} we present some auxiliaries concerning
  the validity 
of various hypotheses made throughout the paper: such as e.g.  the existence of a strong Markov solution of the
martingale problem and the existence
of the principal eigenfunction corresponding to $A$.

\section{Preliminaries}
\label{sec2}

\subsection{Basic notation}
% Throughout the paper we use the following notation. For $ (E,{\rm d})$
% a locally compact separable metric space, denote by

Suppose that $B$ is an arbitrary set. For functions $f,g:B\to[0,+\infty)$
we write $f\preceq g$ on $B$ if there exists number $C>0$, i.e., \textit{constant}, such that
$$
f(x)\le Cg(x),\quad x\in B. 
$$
Furthermore, we write $f\sim g$ if $f\preceq g$ and $g\preceq f$. 
Throughout the paper    we denote 
\begin{equation*}
%\label{bfB}
\bar f_B:=\sup_{x\in B}f(x)\quad\mbox{and}\quad {\underline f}_B:=\inf_{x\in B}f(x).
\end{equation*}

For a metric space $E$ we denote by  $\BB(E)$ its Borel
$\si$-algebra.
Given a subset $D\subset E$ we let $D^c:= E\setminus D$ be its
complement and ${\rm cl}\, D$
be its closure. Let
$B_b(E)$ ($B^+_b(E)$) be the space of all (non-negative) bounded Borel
measurable functions and let $C_b(E)$ ($C_c(E)$) be the space of all
bounded  continuous (compactly supported) functions on $ E$.
Furthermore by ${\cal M}(E)$ we denote the set of all Borel positive measures
on $E$. 
Suppose that $\mu,\nu\in {\cal M}(E)$. We say that $\mu$ dominates
$\nu$ ($\nu$ is absolutely continuous with respect $\mu$) and write $\nu\ll\mu$ if all null sets
for $\mu$ are also null for $\nu$.
The measures are equivalent and write $\mu\sim\nu$ if
$\mu\ll\nu$ and $\nu\ll\mu$.

Given a point $x\in E$ and $r>0$ we 
let $B(x,r)$ be the open ball of radius $r$ centered at $x$ and 
$\bar B(x,r)$ its closure.
As it is customary
for a given function $f:E\to\R$ we denote  $\|f\|_\infty=\sup_{x\in
  E}|f(x)|$. 
For   $\nu$  -
  a  signed Borel measure on $E$ - we define its total variation
  norm   as
\begin{equation}
\label{TV}
\|\nu\|_{\rm TV}:=\sup_{\|f\|_{\infty}\le 1}\left|\int_E fd\nu\right|.
\end{equation}

If $D\subset \R^d$ is open we 
let $C^k(D)$, $k\ge1$ be the class of $k$-times
continuously differentiable functions in $D$. By $C_0(D)$ we denote
the subset of  $C(D)$ that consists of functions extending
continuously to ${\rm cl}\, D$ by letting $f(x)\equiv0$, $x\in\partial D$ -
the boundary of $D$.

Throughout the paper, if it is not stated otherwise,  $m$ shall denote  any non-trivial positive $\sigma$-finite Borel measure on $\BR^d$.
We let $\ell_d$ be the $d$-dimensional Lebesgue measure on
$\bbR^d$ and $dx$ be   the respective volume element.

For $p\in[1,+\infty)$ we denote by  $L^p(D)$  ($L^p_{\rm loc}(D)$) the space of functions
that are integrable with their $p$-th power on 
$D$ (any compact subset of $D$) with respect to $\ell_d$. As usual
$L^\infty(D)$ is the space of functions with a finite essential
supremum norm. {By  $W^{k,p}(D)$
($W^{k,p}_{\rm loc}(D)$) we denote the Sobolev space of
functions whose $k$ generalized derivatives belong to $L^p(D)$
($L^p_{\rm loc}(D)$). Finally, let  $W^{k,p}_{0}(D)$ be the closure of
$C_0^\infty(D)$ in the $W^{k,p}(D)$ - norm.}

\subsection{Second-order, elliptic integro-differential operators}

{
{Suppose that $D$ is an open set and $A$,  defined by \eqref{A},
  satisfies A1) and A2).
We let
\begin{equation}
\label{MA}
M_A:=\sum_{i,j=1}^d\|q_{i,j}\|_\infty+\sum_{i=1}^d\|b_{i}\|_\infty+N_*<+\infty.
\end{equation}
Obviously in order to give meaning to $Au(x)$ for $x\in D$ it suffices only to
assume that $u\in C^2(D)\cap C_b(\bbR^d)$. In fact, we can define
$Au$ as an element of $L^p(D)$ even if  $u\in W^{2,p}_{\rm loc}(D)\cap
C_b(\bbR^d)$ when $p>d$. 
This is possible due to a well known estimate, see \cite[Lemme 1,
p. 361]{bony}:  for any $p>d$ there
exists $C>0$ such that
$$
\|U[u]\|_{L^p(\bbR^d)}\le C\|\nabla^2 u\|_{L^p(\bbR^d)},\quad u\in W^{2,p}(\bbR^d),
$$
where
$$
U[u](x):=\sup_{|y|>0}|y|^{-2}\left|u(x+y)-u(x)-\sum_{i=1}^dy_i\partial_{x_i}u(x)\right|.
$$}

}
\bigskip
Given a bounded and measurable function $c:D\to\bbR$ we let
\begin{equation}
\label{c-c}
\bar c_D=\sup_{x\in D} c(x),\quad \underline{c}_D:=\inf_{x\in D}
c(x)\quad\mbox{and}\quad \langle c\rangle_{D,m}:=\int_Dc(x)\,m(dx).
\end{equation}
Concerning hypotheses made about $c$.
In some of  the results we shall require a stronger
condition  than A3). Namely, we suppose either
\begin{itemize}
\item[A3')] 
$c\in { B}_b^+(D)$ and 
$
\langle c\rangle_{D,m}>0,
$ (i.e. $c\not\equiv 0$, $m$-a.e.)
\end{itemize} 
or   an even  stronger assumption
\begin{itemize}
\item[A3'')] 
$c\in { B}_b^+(D)$ and 
$
\underline{c}_D>0.
$
\end{itemize}

\label{sec4}

 \subsection{Strong Markovian solution to a martingale problem associated with operator $A$} 

\label{SMM}

Suppose that $\partial\not\in\bbR^d$. 
Consider the space
$\bar\bbR^d:=\{\partial\}\cup\bbR^d$ with the topology of the one
point 
compactification of $\bbR^d$ by $\partial$. Any function
$f:\bbR^d\to\bbR$ can be extended to $\bar\bbR^d$ by letting $f(\partial)=0$. Let $\bar{\cal D}$ be the
space consisting of all 
functions $\om: [0,+\infty)\to\bar\BR^d$, that are right continuous and
possess the left limits for all $t\ge0$ (c\'adl\'ags), equipped with
the Skorochod topology, see e.g. Section 12 of \cite{bil}. 
Define the canonical process $X_t(\om):=\om(t)$, $\om \in \bar{\cal D}$
and its natural filtration $({\cal F}_t)$, with ${\cal F}_t:=\si\left(X_s,\,0\le s\le t\right)$.
Point $\partial$ is called
the {\em cemetery state} of the process.
Let ${\cal D}:=D([0,+\infty);\BR^d)$ be the subset of 
$\bar{\cal D}$ consisting of all 
$\om: [0,+\infty)\to\BR^d$.
Given $t\ge0$ define the {\em shift operator} $\theta_t:\bar{\cal D}\to
\bar{\cal D}$ by $\theta_t(\om)(s):=\om(t+s)$, $s\ge0$.

\begin{definition}[A solution of the martingale problem associated
  with  $A$]
\label{df020311-20}
Suppose that $\mu$ is a Borel probability measure on $\bbR^d$.  A
Borel  probability 
measure $P_{\mu}$ on $\bar {\cal D} $ is called a {\em solution of the martingale problem}  associated
  with $A$ with the initial distribution $\mu$ if
\begin{itemize}
\item[i)] $P_{\mu}[X_0\in Z]=\mu[Z]$ for any Borel measurable
 $Z\subset \bbR^d$.
\item[ii)] For every $f\in C_b^2(\BR^d)$
- a $C^2$-smooth function bounded with its two derivatives on $\R^d$ - the process
\begin{equation*}
%\label{Mtf}
M_t[f]:= f(X_t)-f(X_0)-\int_0^tAf(X_r)\,dr,\, t\ge 0
\end{equation*}
is a  (c\'{a}dlag)  martingale under measure $P_\mu$ with respect to natural filtration
$({\cal F}_t)_{t\ge0}$ generated by the canonical process. 

\item[iii)] $P_\mu[{\cal D}]=1$.
\end{itemize}
\end{definition}

As usual we
write $P_x:=P_{\delta_x}$,  $x\in\bbR^d$ and say that $x$ is the
initial condition. We also let $P_\partial=\delta_\partial$. The expectations with respect to $P_{\mu}$ and
$P_x$ shall be denoted by $\E_{\mu}$ and $\E_x$, respectively.

\begin{definition}[A strong Markovian solution of the martingale problem]
We say that a family of Borel probability measures
  $(P_x)_{x\in\bbR^d}$   on
  $\bar{\cal D} $ is  a strong Markovian solution to the martingale problem  associated with $A$
if:
\begin{itemize}
\item[i)] each $P_x$ is a solution of the martingale problem
  associated with $A$, corresponding to the
initial condition at $x$,

\item[ii)] the canonical  process $(X_t)$ is 
  strongly Markovian with respect to the natural filtration $({\cal
    F}_t)$ and the family $(P_x)_{x\in\bbR^d}$,
\item[iii)] the mapping $x \to P_x[C]$ is measurable for any Borel
  $C\subset{\cal D}$,
\item[iv)] for any Borel probability measure $\mu$ on $\bbR^d$  the probability measure
 $$
P_{\mu}(\cdot) :=\int_{\bbR^d}
P_x (\cdot) \mu(dx)
$$
is a solution to the martingale problem associated with $A$ with the
initial distribution $\mu$.
\end{itemize}
\end{definition}

\bigskip

Our hypothesis concerning the
martingale problem can be formulated as follows.
\begin{itemize}
\item[A4)] 
The martingale problem associated with the operator
$A$ admits a strong Markovian solution.
\end{itemize}

We  discuss   conditions  sufficient  for the validity of A4) in Section \ref{sec.a4}.

\bigskip

\subsection{Analytic description of the canonical process}

\subsubsection{Exit time, gauge function, transition probability semigroup and
  resolvent operator} 

\label{ETS}

Suppose that the operator $A$ satisfies condition A4). 
For a given domain $D$ define 
the exit time $ {\tau_D:{\bar {\cal D}}\to[0,+\infty]}$ of the canonical process $(X_t)_{t\ge0}$ from $D$ as
\begin{equation}
\label{tau-D}
\tau_D:=\inf[t>0:\,X_t\not\in D].
\end{equation}
It is a stopping time,
i.e. for any
$t\ge0$ we have $[\tau_D\le t]\in {\cal F}_t$, see Theorem I.10.7, p. 54 of
\cite{bg} and Theorem IV.3.12, p. 181 of \cite{ethier-kurtz}. 
%\textcolor{red}{By
%$\tau_{\partial}$ we denote the stopping time corresponding to
%$D=\bbR^d$ (i.e. the time of reaching the cemetery state).}

We formulate  the following hypothesis:
\begin{itemize}
\item[ET)] the exit time from $D$ is a.s. finite, i.e.
\begin{equation}
\label{012603-20}
P_x[\tau_D<+\infty]=1,\quad x\in D.
\end{equation}
\end{itemize}

\begin{remark}
\label{rmk010211-20}
It turns out that uniform ellipticity of $\bf Q$ (cf. \eqref{la-K})  implies ET)
for a bounded domain $D$, see e.g.  \cite[Lemma 4]{MoFo}. In many
cases however we can also verify it without assuming
 {uniform ellipticity condition}, see Section  \ref{rm3.14} below for a
 more detailed discussion.  
\end{remark}

\begin{definition}[The gauge function for $c(\cdot)$ and domain $D$]
\label{df4.2}
The function
\begin{equation*}
%\label{eq2.pvf}
v_{c,D}(x)=\E_xe_{c}(\tau_D),\quad x\in \BR^d,
\end{equation*}
where 
\begin{equation*}
%\label{ecD}
e_{c}(t):=\exp\left\{-\int_0^{t}c(X_r)\,dr\right\},\quad t\ge0,
\end{equation*}
is called the {\em gauge function} corresponding to $c(\cdot)$ and domain $D$, cf Section 4.3 of \cite{cz}.
\end{definition}
Obviously, if  $c$ satisfies $A3)$, then $0\le v_{c,D}\le 1$.
Let us denote then
\begin{equation}
\label{wcD}
w_{c,D}(x):=1-v_{c,D}(x),\quad x\in \BR^d.
\end{equation}
Obviously  $w_{c,D}(x) \ge0$, $x\in \BR^d$. 
We postpone a more detailed discussion of properties of
$w_{c,D}$ until Section \ref{sub.sub12} below. 

Define the transition semigroup generated by
operator (\ref{A}) on $D$ with the null exterior condition
\begin{equation}
\label{PDt}
P^D_tf(x):=\E_x\left[f(X_t), \, t<\tau_D\right]\quad t\ge 0,\quad f\in B_b(D).
\end{equation}
\begin{definition}
\label{exces}
Suppose that $\al\ge0$.
A Borel function $f:E\to[0,+\infty)$ is called {\em $\al$-excessive}, if
$P_t^Df(x)\le e^{\al t}f(x)$, $t\ge0$ and $\lim_{t\to0+}P_t^Df(x)=f(x)$ for all
$x\in E$. When $\al=0$,  the function is simply called {\em excessive}.
\end{definition}

We define the resolvent of
$A$ on $D$ for any non-negative
  $f\in B(D)$ and $\al\ge 0$ by letting
\begin{equation}
\label{RDt}
R^D_\alpha f(x):=\int_0^\infty e^{-\alpha t}P^D_tf(x)\,dt =\E_x\left[\int_0^{\tau_D}e^{-\alpha t}f(X_t)\,dt\right],\quad x\in D.
\end{equation}
Set $R^D:=R^D_0$. The definition of $R^D_\alpha$   obviously extends to 
$f\in B_b(D)$, when $\alpha> 0$.

 The operators  $(R^D_\al )_{\al\ge 0}$ satisfy the resolvent
identity
\begin{equation}
\label{res-id}
R^D_\al-R^D_\beta=(\beta-\al)R^D_\al R^D_\beta,\quad \al,\beta\ge 0.
\end{equation}
We shall denote by $R^D_\alpha (x,\cdot)$ the Borel measure on $D$,
 defined by 
$$
R^D_\alpha (x,B)=R^D_\alpha 1_B(x), \quad B\in {\cal B}(D),\,
 x\in D.
$$

Suppose that $m$ is a Borel measure on $D$ and $t\ge0$. We
  define the transfer measure
$mP_t(B):=\int_{D}P_t1_Bdm$, $B\in {\cal B}(D)$. An analogous notation
shall be used in the case of the resolvent family
$\big(R^D_\alpha\big)$. Measure $m$ is called {\em excessive} if
$mP_t\le m$ for all $t\ge0$.

It will be convenient for us to work sometimes with measures $P_x^c$, $x\in D$
defined on the  path space $\bar{\cal D}$, cf Section \ref{SMM}, that
correspond to the process
$(X_t)$ killed at rate $c(X_t)$, see \cite[Section III]{bg}. We shall denote by $(P_t^{c,D})_{t\ge0}$ and $(R^{c,D}_\al)_{\al\ge0}$
the corresponding  semigroup and the resolvent family, determined by formulas
analogous to \eqref{PDt} and \eqref{RDt}, with respect to $P_x^c$. 

\subsection{Weak subsolution of
  $(-A+c)v=g$}

\label{sec2.5}

%\textcolor{red}{\em \large{We need a reference here}}

%\textcolor{blue}{\bf \large{DOTAD}}

\begin{definition}[Weak subsolution, supersolution and solution]
\label{df3.1g}
Suppose that $c(\cdot)$ satisfies the hypothesis A3), $D$ is open and $g(\cdot)$ is a
Borel measurable function on $D$ such that $g\le0$.
A function $u\in  B_b(\BR^d)$ is called a {\em weak subsolution}
of the equation
\begin{equation} 
\label{eq3.1a}
(-A+c)v(x)=g(x),\quad x\in D,
\end{equation} 
if 
\begin{equation}
\label{subsol.g}
u(x)\le \mathbb E_x\left[e_c(\tau_D\wedge t)u(X_{\tau_D\wedge t})\right]+\mathbb E_x\left[\int_0^{\tau_D\wedge t}e_c(r)g(X_r)\,dr\right] \, \mbox{for all } x\in D,\, t\ge 0
\end{equation}
 {for some solution of the martingale problem associated with
  $A$. Throughout the remainder of the paper we assume that  the probability
  measure appearing in \eqref{subsol.g} (thus also the
  respective semigroup   and the resolvent family) is fixed.}

 Denote by  $\mathcal U_c(g)$    the set
of all   weak  
subsolutions to \eqref{eq3.1a} and by $\mathcal U_c^+(g)$ its subset
consisting  of those
$u  $, for which   $\bar u_{D_S }\ge0$.   We let  $\mathcal
  U_c=\mathcal U_c(0)$ and  $\mathcal
  U_c^+=\mathcal U_c^+(0)$.

We say that $u$ is a {\em weak supersolution} of \eqref{eq3.1a}
if the sign $\le$ in the above inequality is replaced by $\ge$. Furthermore $u$ is a  {\em weak solution}
if it is both weak sub- and supersolution. Observe that the regularity of   a subsolution to \eqref{eq3.1a}
is subject to no restriction (besides  boundedness).
\end{definition}

\begin{remark}
The above definition of a weak subsolution coincides
with that of \cite[Definition 3.2, p.  1551]{BL01}.  It is used there
  in the formulation of the  Alexandrov-Bakelman-Pucci type
estimates and maximum principles  for a certain class of generators  of  L\'{e}vy 
processes. 
\end{remark}

\begin{remark}
Note that, if $u\in B_b(D)$ is a subsolution of \eqref{eq3.1a}, then 
  $-R^{c,D}g(x)\ge0 $ is finite for each $x\in D$. In fact $R^{c,D}g\in B_b(D)$ and 
\begin{equation}
\label{021101-21}
\tilde u(x):=u(x)-R^{c,D}g(x),\quad x\in D
\end{equation}
is a subsolution of the homogeneous equation $(-A+c)v(x)=0$.
\end{remark}

Condition  \eqref{subsol.g} can be rewritten using the path measure
$P_x^c$ of the killed process. It reads
\begin{equation}
\label{subsol.g.eq}
u(x)\le \mathbb E^c_x u(X_{\tau_D\wedge t})+\mathbb E^c_x\left[\int_0^{\tau_D\wedge t} g(X_r)\,dr\right] ,\quad \mbox{for all } x\in D,\, t\ge 0.
\end{equation}
Thus,
\begin{equation}
\label{subsol.g.eq1}
u(x)\le P^{c,D}_tu(x)+\mathbb E^c_x [u(X_{\tau_D})\mathbf{1}_{\{t\ge \tau_D\}}],\quad x\in D.
\end{equation}

{
\subsection{Relation between weak   and other types of subsolutions}

\label{svc.w1}
We briefly  discuss the relationship between the
notion of a 
 weak subsolution, introduced in the foregoing,   and some other notions of
soubsolutions that appear throughout the literature, such as: the
{\em classical}, {\em Sobolev} and {\em viscosity subsolutions}. Analogous
 statements can be made about  supersolutions and solutions.

Let us    recall the  definition of a classical subsolution.

\begin{definition}
We say that a function $u$ is a {\em classical subsolution} of  \eqref{eq3.1a} in $D$ if 
 $u\in C^{2}(D)\cap C_b(\BR^d)$ and  
\begin{equation}
\label{subsol}
-Au(x)+c(x)u(x)\le g(x),\quad  x\in D.
\end{equation}
\end{definition}
By  the definition of the martingale problem and It\^o's formula (applied to the product $u(X_t)e_c(t)$)
we can easily deduce that any {\em classical subsolution} is a  {\em
  weak subsolution}  {with respect to any solution of the martingale
  problem associated with operator $A$}.

 \begin{definition} A   {\em Sobolev subsolution} to \eqref{eq3.1a} is a function $u\in W^{2,p}_{\rm loc}(D)\cap
C_b(\bbR^d)$ ($p>d$) such that  \eqref{subsol}  holds $\ell_d$-a.e.
\end{definition}
According to the remark  following \eqref{MA}  $Au$ is well
defined as an element of $L^p(D)$ for any  $u\in W^{2,p}_{\rm loc}(D)\cap
C_b(\bbR^d)$ ($p>d$).
If we assume additionally that  the matrix $ {\bf Q}(x)$ is uniformly elliptic on compacts, then by \cite[Proposition 4.6]{kk01},  \eqref{subsol.g} holds.
Therefore,  under uniform ellipticity condition,  each {\em Sobolev
  subsolution} is  a {\em weak subsolution} of \eqref{eq3.1a}. 

\begin{definition}  A {\em viscosity subsolution} to \eqref{eq3.1a} is
  an upper semi-continuous  function $u\in B_b(\BR^d)$ satisfying:   for any  $x\in D$ and  $\varphi\in C^2_b(\BR^d)$
such that $\varphi(x)=u(x)$ and $\varphi(y)\ge u(y),\, y\in\BR^d$ we have
\[
-A\varphi(x)+c(x)\varphi(x)\le g(x). 
\]
\end{definition}
It goes beyond the scope of this paper to examine  in depth  
the relation between the  weak subsolutions to \eqref{eq3.1a} and the
viscosity subsolution. However,  some sufficient
conditions for viscosity solutions to be  also weak subsolutions can
be formulated. We
postpone a more detailed discussion  untill Section \ref{sec6.5}.

}

\section{The main result - Hopf's lemma}
\label{NS-hopf}

Throughout the remainder of the paper we shall always assume
hypotheses A1) - A4) and ET) without further mentioning them in the subsequent
formulations of the results.

We start with the formulation of the  {\em
  minorization condition}. 
\begin{definition}[Minorization
condition $M(\al, \psi_D^\al, \nu^\alpha_D)$]
\label{minorization}
{Suppose
that $\al\ge0$, $\psi_D^\al:D\to[0,+\infty)$ and  $\nu^\alpha_D$ is a
$\si$-finite Borel measure  on $D$.  
We say the minorization
condition $M(\al, \psi_D^\al, \nu^\alpha_D)$ holds if 
\begin{equation}
\label{RDa}
R^D_\alpha f(x)\ge \psi_D^\al(x) \int_D f d\nu_D^\alpha,\quad x\in
D,\quad f\in B^+_b(D).
\end{equation}}
\end{definition}

The following version of the Hopf lemma holds.
\begin{theorem}[The Hopf lemma]
\label{th3.hl.nuc1}
Recall that $\bar c_D=\sup_{x\in D}c(x)$.  Assume that the  operator
$A$ 
satisfies  
$M(\bar c_D, \psi_D^{\bar c_D}, \nu^{\bar c_D}_D)$ for some strictly
positive function 
$\psi^{\bar c_D} _D:D\to(0,+\infty)$ and  measure $\nu^{\bar c_D} _D$.
\begin{enumerate}
\item[1)]
Suppose that ${\rm supp}\,\nu^{\bar c_D} _D=D$ and
 $u\in \mathcal U^+_c$  is   non-constant  and continuous in $D$.
Then, there exists $a>0$  for which
\begin{equation}
\label{010104-20av1}
\bar u_{D_S }-u(x)\ge a  \psi^{\bar c_D}_{D}(x),\quad x\in D.
\end{equation}
%Moreover, if $u\in C(\overline D)$, then $a,\delta$ do not depend on $\hat x$.
\item[2)]  Suppose that  $m\ll \nu_D^{\bar c_D}$ and 
 $u\in\mathcal U^+_c$  is   non-constant $m$-a.e.   in $D$.
Then, the conclusion of part 1) is in force.
\end{enumerate}
\end{theorem}

The proof of this result is presented below. Due to its length we
divide it into three parts. First, in Section \ref{pfT3.2p} we
introduce some preliminaries needed in the argument. The main step of
the proof is made in Section \ref{pfT3.2m}, modulo some technical
estimate given in \eqref{eq.np1}. The latter is  shown in Section \ref{pfT3.2f}.

\subsection{Preliminaries}

\label{pfT3.2p}

\subsubsection{On the weak maximum principle}
\begin{lemma}
\label{lm.lm.151120}
We have 
\begin{equation}
\label{010793-20dfg1v}
P_x(X_{\tau_D}\in \partial_S D)=1,\quad x\in D.
\end{equation}
\end{lemma}
\begin{proof}
Using the Ikeda-Watanabe formula, see \cite[Remark 2.46,
page 65]{BSW}, we obtain
\begin{equation*}
\begin{split}
&\mathbb E_{x}\left[\sum_{0<s\le\tau_D}\mathbf{1}_{D}(X_{s-})\mathbf{1}_{\mathcal
    S(D)^c\setminus D}(X_s)\right]\\
&
=\mathbb
E_{x}\left[\int_0^{\tau_D}\,ds\int_{\BR^d}\mathbf{1}_{D}(X_{s-})\mathbf{1}_{\mathcal
    S(D)^c\setminus D}
(y)N(X_{s-},dy-X_{s-})\right].
\end{split}
\end{equation*}
From the definition of the range of non-locality $\mathcal S(D)$, see \eqref{cSD}, we conclude
that the right-hand side of the above equation vanishes. Thus,
$X_{\tau_D}\in \mathcal S(D)\setminus D$, if
$X_{\tau_D}\not=X_{\tau_D-}$, or $X_{\tau_D}\in \partial D$, if
otherwise. Hence  \eqref{010793-20dfg1v} follows from the definition
of  $\partial_S D$, see \eqref{exbexc}.
\end{proof}

The following version of the weak maximum principle   is a direct
consequence of Lemma \ref{lm.lm.151120}.

\begin{proposition}[Weak maximum principle for $A$]
\label{prop012904-20}
If $u$  is a weak subsolution to
\eqref{eq3.1a}, then
 \begin{equation}
\label{032804-20apx}
\bar u_D\le   \bar u^+_{\partial_{S}D}.
\end{equation}
When $c\equiv 0$ we get \eqref{032804-20apx}  {with $\bar u^+_{\partial_{S}D}$ replaced by $\bar u_{\partial_{S}D}$}.
% \begin{equation}
% \label{032804-20}
% \sup_{x\in D}u(x)\le \sup_{x\in D^c}u(x).
% \end{equation}
\end{proposition}
%\textcolor{red}{\em Why $u^+$???}

As a corollary to the above proposition we conclude that for any weak
subsolution $u$ to \eqref{eq3.1a} with $\bar u_{D_S }\ge 0$, we have
\begin{equation}
\label{eq.mu.main}
\bar u_{D_S }=\bar u_{\partial_SD}.
\end{equation}

\subsubsection{Some semimartingales associated with the canonical process}
\label{rem.finver1}

Define 
\begin{equation*}
%\label{042001-21}
w:=\bar u_{D_S}-u\ge 0\quad\mbox{in}\,\, D,
\end{equation*}
 where $u$ is a subsolution to \eqref{eq3.1a} such that $\bar
 u_{D_S}\ge 0$. From inequality \eqref{subsol.g.eq1} we infer that 
 \[
 P^{c,D}_tw\le w+\mathbb E^c_x[(-\bar u_{D_S}+u(X_{\tau_D}))\mathbf1_{\{t\ge \tau_D\}}].
 \]
 Therefore,    by Lemma \ref{lm.lm.151120}, we have  
\begin{equation}
\label{011101-21}
P^{c,D}_tw\le w\quad\mbox{ in  }D\quad\mbox{ for any }t\ge 0.
\end{equation}
By \cite[Exercise II.2.17]{bg} 
\begin{equation*}
%\label{032001-21}
\hat w:= \sup_{t>0}P^{c,D}_t w=\lim_{t\to 0^+}P^{c,D}_t w
\end{equation*}
 is an excessive function
with respect to $(P^{c,D}_t)_{t\ge0}$. Thanks to  \eqref{011101-21} we
have
\begin{equation}
\label{011101-21a}
\hat w\le w\quad\mbox{ in  }D.
\end{equation}
Define 
\begin{equation}
\label{hatu}
\hat u:= \lim_{t\to
  0^+}P^{c,D}_tu,\quad  \mbox{in $D$}
\end{equation}
 (the limit exists, since $\lim_{t\to
  0^+}P^{c,D}_t w$ exists). 
Clearly, \begin{equation}
\label{011101-21b}
\hat w=\bar u_{D_S}-\hat u\quad\mbox{ in  }D.
\end{equation}
Therefore, from \eqref{011101-21a} and the fact that $\hat w\ge0$, we have
\begin{equation}
\label{eq.vin.np1}
\hat u(x)\ge u(x),\,\, x\in D.
\end{equation}
From this and \eqref{hatu}
\begin{equation*}
%\label{eq.vin.np115gz}
\sup_D \hat u=\sup_D u.
\end{equation*}
 {We get also, see the comment preceding \cite[Theorem II.3.6]{bg}, that
\begin{equation}
\label{eq.vin.np115gxz}
R^{c,D}_\alpha\hat u(x)=R^{c,D}_\alpha u(x),\quad x\in D,\, \alpha\ge 0.
\end{equation}}
A process $\big(Z_t\big)_{t\ge0}$ is called  {\em additive},
with respect to $(P^{c,D}_x)$, if for any $s,t\ge 0$,
$Z_{t+s}=Z_s+Z_t\circ\theta_s$ a.s., where $\theta_s$ is the shift operator.

\begin{proposition}
\label{prop011101-21}
Suppose that $u\in B_b(D)$ is a weak subsolution of
\eqref{eq3.1a}. Then, there exist a c\`adl\`ag,
increasing, predictable  process $\big(A_t\big)_{t\ge0}$ and
uniformly integrable $P^{c,D}_x$-martingale $\big(M_t\big)_{t\ge0}$, such that for each $x\in D$ we have
\begin{equation}
\label{eq.np0}
\hat u(X_t)=\hat u(X_0)+A_t-\mathbf{1}_{\{t\ge \tau_D\}} \bar u_{D_S}-\int_0^tg(X_r)\,dr+M_t,\quad t\ge 0,\, P^{c,D}_x\mbox{-a.s.}
\end{equation}
Both processes $\big(A_t\big)_{t\ge0}$ and
$\big(M_t\big)_{t\ge0}$ are additive with respect to  
$(P^{c,D}_x)$.
\end{proposition}
\proof
We can assume with
no loss of generality that $g\equiv0$. Otherwise we would consider $\tilde u$, given by \eqref{021101-21}.
By \cite[Theorem III.5.7]{bg}, the process $\big(\hat w(X_t)\big)_{t\ge0}$ is a
right-continuous bounded supermartingale under measure $P^{c,D}_x$ for any
$x\in D$. Let $\hat Y_t:=\hat w(X_t)-\hat w(X_0)$, $t\ge0$.   By the
version of Doob-Meyer decomposition, see \cite[Theorem 3.18]{CJPS},
 we conclude that there exist a c\`adl\`ag
increasing predictable  process $\big(A_t\big)_{t\ge0}$ and
uniformly integrable martingale $\big(M_t\big)_{t\ge0}$ under measure
$P^{c,D}_x$ such that $\hat Y_t=M_t-A_t$, $t\ge0$. According to ibid. both processes $\big(A_t\big)_{t\ge0}$ and
$\big(M_t\big)_{t\ge0}$  are additive with respect to
$\big(P^{c,D}_x\big)$. Therefore their choices do not  depend on $x\in D$. Formula \eqref{eq.np0}
follows directly from \eqref{011101-21b}.\qed

\bigskip

Let 
\begin{equation*}
%\label{Yt}
Y_t:= \hat u(X_t)+ \mathbf{1}_{\{t\ge \tau_D\}} \bar u_{D_S}
,  \quad t\ge 0.
\end{equation*}
 Formula \eqref{eq.np0} can be rewritten as follows
\begin{equation}
\label{eq.np0vy}
Y_t=\hat u(X_0)+A_t-\int_0^tg(X_r)\,dr+M_t,\quad t\ge 0,\, P^{c,D}_x\mbox{-a.s.}
\end{equation}
 Since $\big(M_t\big)_{t\ge0}$, $\big(A_t\big)_{t\ge0}$ are additive,  with respect to
  $(P^{c,D}_x)$, so is the process
$\big(Y_t\big)_{t\ge0}$. Since $\hat w(\partial)=0$  we have 
$A_t=A_{t\wedge\tau_D}$, $t\ge0$, $P^{c,D}_x$ a.s. The same also holds for $\big(M_t\big)_{t\ge0}$ and $\big(Y_t\big)_{t\ge0}$.

\subsubsection{Some properties of $\hat u$}

Recall that
$\underline u_D:=\inf_Du$ and $\bar u_D:=\sup_Du$. If $m$
is a $\si$-finite Borel measure on $D$ we let
$$
\underline u_{D,m}:=\essinf_Du, \qquad \bar u_{D,m}:=\esssup_Du,
$$ 
where the respective essential supremum and infimum correspond  to the
 measure $m$. 
In the case $\hat u$ is defined by \eqref{hatu} we define
\begin{equation*}
%\label{022001-21}
\underline{\hat u}_{D,m}:=\essinf_D\hat u, \qquad 
\overline{\hat u}_{D,m}:=\esssup_D\hat u.
\end{equation*}
The following result shall be useful in the proof of  Theorem
\ref{th3.hl.nuc1}.
\begin{proposition}
\label{prop012001-21}
Suppose that $u$ is a weak subsolution of \eqref{eq3.1ab}. If $u\in
C_b(\bbR^d)$, then
$\hat u=u$. Furthermore, assume that $m$ is a $\si$-finite Borel measure
and $u$ is a bounded measurable subsolution such that $\underline
u_{D,m}<\bar u_{D,m}$.  Suppose   that the operator $A$
satisfies  the minorization condition
$M(\bar c_D, \psi_D^{\bar c_D}, \nu^{\bar c_D}_D)$ for some
 $\psi^{\bar c_D} _D:D\to[0,+\infty)$, that is not identically equal $0$,
 $m$-a.e.  and    $m\ll \nu^{\bar c_D} _D$.
 Then
\begin{equation}
\label{012001-21}
\underline{\hat u}_{D,m}<\overline{\hat u}_{D,m}.
\end{equation}
\end{proposition}
\proof
The first part of the proposition concerning a continuous subsolution is obvious,   thanks to the continuity of $X_t$ at $t=0$.
To prove \eqref{012001-21},  {we show that $\hat u=u$,
  $m$-a.e. on $D$. Suppose that $m(\{\hat u>u\})>0$.
Then $\nu^{\bar c_D}_D(\{\hat u>u\})>0$.
Applying
condition $M(\bar c_D, \psi_D^{\bar c_D}, \nu^{\bar c_D}_D)$ we
obtain
$$
R_{\bar c_D}^{c,D} (\hat u-u)(x)\ge \psi_D^{\bar c_D}(x) \int_D(\hat u-u)d\nu^{\bar c_D}_D,\quad x\in D,
$$
which contradicts  \eqref{eq.vin.np115gxz}.}\qed

\subsection{Main step in the proof of Theorem  \ref{th3.hl.nuc1}}
\label{pfT3.2m}
We carry out the proof of both parts of the theorem  simultaneously.
The main point of the proof is to show that there exists an excessive function $v$,
with respect to $(P^{c,D}_t)$, such that
\begin{equation}
\label{eq.np1}
\bar u_{D_S}-u(x)\ge v(x),\quad x\in D\quad\mbox{and }  \nu^{\bar c_D}_D(x:\,v(x)>0)>0,
\end{equation}
where $\nu^{\bar c_D}_D$ is the measure appearing in condition $M(\bar c_D, \psi_D^{\bar c_D}, \nu^{\bar c_D}_D)$.
Once this  is done, we construct an appropriate approximation of $v(x)$,
see \eqref{080309-20}  below, which can be bounded from below using the
minorization condition \eqref{RDa} and the conclusion of the theorem follows.

\subsubsection*{Conclusion of  the proof of the theorem under \eqref{eq.np1}}

Let us suppose that \eqref{eq.np1} holds    and we show  then how to conclude the proof of the theorem.
We construct an approximating sequence
of non-negative functions $(f_k)_{k\ge1}$ such that
$(R^{c,D}f_k(x))_{k\ge1}$
is increasing and
\begin{equation}
\label{080309-20}
\lim_{k\to+\infty}R^{c,D}f_k(x)=v(x),\quad x\in D.
\end{equation}
The sequence $(f_k)$ is constructed by  taking the Yosida approximation
\begin{equation*}
%\label{fk}
f_k(x) :=k\big(v (x)-kR^{c,D}_k v (x)\big), \quad x\in D,\, k\ge 1.
\end{equation*} 
By  \cite[(2.6) of
Chapter II, p. 73]{bg} we have $R^{c,D}f_k(x)=kR^{c,D}_k v (x),\, x\in D$.
Therefore, by virtue of
\cite[(2.4) of Chapter II, p. 73]{bg}, the sequence
$(R^{c,D}f_k(x))_{k\ge1}$ is monotone increasing   for each $x\in D$
fixed  and by Proposition II.2.6 of ibid. \eqref{080309-20} holds.
Using the monotonne convergence theorem, we conclude from
\eqref{080309-20} that for each $n\ge1$,
\begin{align}
\label{030409-20}
\nonumber \lim_{k\to+\infty}\mathbb E^c_x\left[\int_0^{\tau_D\wedge n} f_k(X_s)\,ds\right]&=
\lim_{k\to+\infty} R^{c,D}f_k(x)-P^{c,D}_nR^{c,D}f_k(x)\\&=v (x)-P^{c,D}_n v (x)=: v_n(x),\quad x\in D.
\end{align}
Obviously $(v_n(x))_{n\ge1}$ is increasing for each $x\in D$ fixed and
\begin{equation*}
%\label{060409-20}
\lim_{n\to+\infty}v_n(x)=v (x), \quad x\in D.
\end{equation*}
Therefore, by the second inequality in \eqref{eq.np1}, there exists $N\ge 1$ such that  $\int_D v_N d\nu_D^{\bar c}>0$.

Substituting $R^D_{\bar c_D+1}f$
in place of $f$ in \eqref{RDa}, with $\al=\bar c_D$, we get, by the
resolvent identity \eqref{res-id}, that
\begin{equation}
\label{RDa14da}
R^D_{\bar c_D} f(x)\ge R^D_{\bar c_D} R^D_{\bar c_D+1}f(x)\ge\psi_D^{\bar c_D}(x) \int_D R^D_{\bar c_D+1}f 
d\nu_D^{\bar c_D},\quad x\in D,\quad f\in B^+_b(D).
\end{equation}
From 
\eqref{RDa14da} we conclude that
\begin{equation}
\label{020409-20}
R^{c,D} f_k(x)\ge  R^D_{\bar c_D} f_k(x)\ge \psi_D^{\bar c_D}(x) \int_D R^{D}_{\bar c_D+1}f_k d\nu_D^{\bar c_D},\quad x\in D.
\end{equation}
Moreover, for each $N\ge1$
\begin{align*}
& R^{D}_{\bar c_D+1}f_k(x)= \mathbb E_x\left[\int_0^{\tau_D}e^{-s(\bar
                        c_D+1)}f_k(X_s)\,ds\right]\\
&
\ge  \mathbb E_x\left[\int_0^{\tau_D\wedge N}e^{-s(\bar c_D+1)}e^{-\int_0^sc(X_r)\,dr}f_k(X_s)\,ds\right]
\ge  e^{-N(\bar c_D+1)} \mathbb E^c_x\left[\int_0^{\tau_D\wedge N}f_k(X_s)\,ds\right].
\end{align*}
Substituting into \eqref{020409-20} we get
\[
R^{c,D}f_k(x) \ge e^{-N(\bar c_D+1)}   \psi_D^{\bar c_D}(x) \int_D \mathbb E^c_y\left[\int_0^{\tau_D\wedge N}f_k(X_s)\,ds\right] \nu_D^{\bar c_D}(dy),\quad x\in D.
\]
Letting $k\rightarrow \infty$, we get, cf \eqref{030409-20},
\begin{equation}
\label{040409-20}
v(x) \ge b_N \psi_D^{\bar c_D}(x),\quad x\in D,
\end{equation}
where 
$$
b_N:=e^{-N(\bar c_D+1)}  \int_D v_N d\nu_D^{\bar c_D}>0.
$$
From \eqref{040409-20} and \eqref{eq.np1} we conclude that  \eqref{010104-20av1} holds.

\subsection{The proof of  \eqref{eq.np1}}

\label{pfT3.2f}

It suffices to show the existence of an excessive function $v$
such that
\begin{equation}
\label{eq.np1a}
\bar u_{D_S}-\hat u(x)\ge v(x),\quad x\in D\quad\mbox{and }  \nu^{\bar c_D}_D(x:\,v(x)>0)>0
\end{equation}  
(note   $\hat u$  in the first inequality). Estimate \eqref{eq.np1}  then   follows  from \eqref{eq.vin.np1}.
%and the weak maximum principle \eqref{032804-20}. 

Observe   that by \eqref{eq.np0}, \eqref{012603-20}
 and the
  Fatou lemma
\begin{equation}
\label{011901-21}
\bar u_{D_S}-\hat u(x)\ge \mathbb E^{c,D}_xA_{\tau_D},\quad x\in D.
\end{equation}
%where 
%$$
%A^g_t:=A_t-\int_0^tg(X_r)\,dr,\quad t\ge 0
%$$
%$\big(A_t\big)_{t\ge0}$ is as in \eqref{eq.np0}.
Since $\big(A_t\big)_{t\ge0}$ is additive, increasing, right-continuous and
$A_0=0$, the function 
\begin{equation}
\label{wgA}
v^A(x):=\mathbb E_x^{c,D}A_{\tau_D}
\end{equation}
is excessive  with respect to $(P^{c,D}_t)$. If 
\begin{equation*}
%\label{052001-21}
\nu^{\bar c_D}_D(x:\,v^A(x)>0)>0,
\end{equation*}
 then we  get \eqref{eq.np1a}, with $v(x):=v^A(x)$,  as a
conclusion 
from \eqref{011901-21}.

%We shall argue that \eqref{052001-21} holds by contradiction. 
Suppose  now that 
\begin{equation}
\label{012101-21a}
\nu^{\bar c_D}_D(x:\,v^A(x)>0)=0.
\end{equation} 
Since $\hat u$ is   non-constant in
$D$ in part 1) of the theorem and  non-constant in
$D$,  $m$-a.e. in part 2),  we
have $\underline{\hat u}_{D,m}<\overline{\hat u}_{D,m}$ in both  cases (see Proposition \ref{prop012001-21}). Suppose that  $\varepsilon>0$ and $a\in\BR$ are
such that 
\begin{equation}
\label{050309-20}
\underline{\hat u}_{D}\le\underline{\hat u}_{D,m} <a-2\varepsilon<a<a+2\varepsilon< \overline{\hat u}_{D,m}\le\overline{\hat u}_{D}=\bar u_D.
\end{equation}
The last equality follows from \eqref{eq.vin.np1}.
Let $\varphi:\BR\rightarrow \BR$ be smooth, nondecreasing, convex
and such that 
\begin{equation}
\label{phi}
 \varphi(x)=a-\varepsilon,\, x\le a-\varepsilon\quad\mbox{
and}\quad \varphi(x)=x,\, x\ge a+\varepsilon.
\end{equation}
%We let $S_t=[u]_{{\rm cl}\,D\cup\mathcal S(D)}\mathbf{1}_{\{t\ge \tau_D\}},\, t\ge 0$.
By the It\^o formula, see \cite[III Theorem 7.32, p.78]{Protter} applied to \eqref{eq.np0vy}, for any
smooth, {nondecreasing,} convex $\varphi:\BR\rightarrow \BR$,
\begin{align}
\label{RDa14dat9-1-21bv}
\nonumber
  \varphi(Y_{t\wedge \tau_D} )=\varphi(\hat u(X_0))&+\int_0^{t\wedge \tau_D}\varphi'(Y_{s-})\,dA_s+\int_0^{t\wedge \tau_D}\varphi'(Y_{s-})\,dM_s\\
&
+\frac12 \int_0^{t\wedge \tau_D} \varphi''(Y_s)\,d[M]^c_s+J_{t\wedge
  \tau_D},\quad t\ge0,\,  P^{c,D}_x\mbox{-a.s.}
\end{align}
Here  $([M]^c_t)_{t\ge0}$ is the continuous part of the quadratic
variation process $([M]_t)_{t\ge0}$  associated with  martingale $(M_t)_{t\ge0}$. The jump part of $([M]_t)_{t\ge0}$
is given by
\begin{equation}
\label{eq.jump1287}
  J_t:=\sum_{0<s\le
  t}\Big\{\varphi(Y_s)-\varphi(Y_{s-})-\varphi'(Y_{s-})(Y_s-Y_{s-})\Big\},\quad 
t\ge0,\, P^{c,D}_x\mbox{-a.s.}
\end{equation}
By convexity of $\varphi$ each term in the above sum is
non-negative. Thus, $(J_t)_{t\ge0}$ is an increasing c\`adl\`ag process satisfying
$J_0=0$.
Since all the processes appearing in the right-hand side
of \eqref{eq.jump1287} are additive (see Section
\ref{rem.finver1}), the process  $(J_t)_{t\ge0}$ is also additive.

Let $(J^p_t)_{t\ge0}$ be the dual predictable projection of $(J_t)_{t\ge0}$, with respect to
$P^{c,D}_x$, i.e. a predictable  increasing c\`adl\`ag process such that 
\begin{equation}
\label{Nt}
N_t:=J_t-J_t^p, \quad t\ge0
\end{equation} 
is
a martingale.  By \cite[Theorem 3.18]{CJPS}, $J^p$ exists and   is additive.
Define the process
\begin{equation*}
%\label{Kt}
K_t:=\frac12 \int_0^{t} \varphi''(Y_s)\,d[M]^c_s+J^p_t,\quad t\ge0.
\end{equation*}
Since all the processes appearing in the definition are additive, so
is  $(K_t)_{t\ge0}$. In addition, it is  c\`adl\`ag, increasing and $K_0=0$. We
infer therefore that 
$v^K(x):= \mathbb E^{c,D}_xK_{\tau_D}$ is an excessive function with respect to $(P^{c,D}_t)$.
Applying 
expectation $\mathbb E^{c,D}_x$ to both sides of
\eqref{RDa14dat9-1-21bv} and taking the limit $t\to+\infty$ we obtain
\begin{align}
\label{1-21bvas}
\varphi(\bar u_{D_S})\ge \varphi(\hat u(x))+\mathbb
  E^{c,D}_x\left[\int_0^{\tau_D}\varphi'(Y_{s-})\,dA_s\right]
+v^K(x),\quad
  x\in D.
\end{align}
By the definition of function $\varphi$, cf \eqref{phi}, we have 
\begin{equation*}
%\label{fiv}
\varphi(y)\ge y,\, y\in \BR\quad\mbox{ and }\quad\varphi(\bar u_{D_S})=
\bar u_{D_S}.
\end{equation*}
This combined  with \eqref{1-21bvas} gives the first inequality of
\eqref{eq.np1a}, with $v:=v^K$. We shall  prove   the second inequality
of \eqref{eq.np1} arguing by contradiction. 
Suppose that 
\begin{equation}
\label{012105-21}
v^K(x)=0\quad\mbox{  for  $\nu^{\bar
  c_D}_D$-a.e. $x\in D$.}
\end{equation}
 Since  we have also assumed  that $\nu^{\bar
  c_D}_D(v^A>0)=0$, therefore (cf \eqref{wgA})
\begin{equation}
\label{eq.xc3.m12.1}
P^{c,D}_x\Big(K_{\tau_D}=0\Big) =1\quad\mbox{and}\quad  P^{c,D}_x\Big(A_{\tau_D}=0\Big)=1,\quad \text{for $\nu^{\bar c}_D$-a.e. $x\in D$}.
\end{equation}
The first equality implies that
\begin{equation}
\label{Kt1}
\frac12 \int_0^{t\wedge \tau_D} \varphi''(Y_s)\,d[M]^c_s=-J^p_{t\wedge
  \tau_D},\quad t\ge0,\quad P^{c,D}_x\,\mbox{-a.s.}
\end{equation}
Combining \eqref{RDa14dat9-1-21bv}, 
\eqref{eq.xc3.m12.1} and \eqref{Kt1} we conclude that (cf \eqref{Nt})
\begin{equation}
\label{RDa14dat9-1-2}
 \varphi(Y_{t\wedge \tau_D})=\varphi(\hat u(X_0))+\int_0^{t\wedge \tau_D}\varphi'(Y_{r-})\,dM_r+N_{t\wedge \tau_D}, \quad t\ge0,\, P^{c,D}_x\mbox{-a.s.}
\end{equation}
for $\nu^{\bar c_D}_D$-a.e. $x\in D$.

Let 
$$
{\cal N}:=[x\in D:\, v^K(x)>0,\quad\mbox{or}\quad v^A(x)>0].
$$ 
We have 
\begin{equation}
\label{022105-21}
\nu^{\bar c_D}_D\big({\cal N}\big)=0,
\end{equation}
  by virtue of \eqref{012105-21} and \eqref{012101-21a}.
As we have already shown above, thanks to \eqref{RDa14dat9-1-2} and
the fact that $\hat u(\cdot)$ is bounded, the process  
$
M^\varphi_t:=\varphi(Y_{t\wedge \tau_D})$,
$t\ge 0$  is a bounded $(\FF_t)$-martingale under measure $P^{c,D}_x$ for  each
$x\in D\setminus{\cal N}$.
Obviously, from \eqref{phi}, 
\begin{equation}
\label{040309-20}
M^\varphi_t\ge a-\varepsilon,\quad  t\ge0.
\end{equation}
By \eqref{050309-20}, \eqref{022105-21}, and  either the continuity of $\hat u$ and  the fact
that $\nu^{\bar c_D}_D$ has full  support (in the case of part 1) of the theorem)  or   the fact  that   $m\ll \nu^{\bar c_D}_D$
 (in the case of part 2) of the theorem), we infer 
$$
\nu^{\bar c_D}_D\Big((D\setminus{\cal N})\cap \{\hat u<a-\varepsilon\}\Big)>0.
$$
Suppose now that $\bar x\in \{\hat u<a-\varepsilon\}\cap\Big( D\setminus{\cal
  N}\Big)$. Then, $M^\varphi_0=a-\varepsilon$, $P^{c,D}_{\bar x}$-a.s. 
Since for this particular choice of $\bar x$ the process $(M^\varphi_t)$  is a
 martingale under measure $P^{c,D}_{\bar x}$, thanks to
 \eqref{040309-20} we have, see e.g. \cite[VI Theorem 3.13, page 371]{Protter},
\begin{equation*}
%\label{Mfi}
M^\varphi_t=a-\varepsilon,\quad t\ge 0,\, P_{\bar x}^{c,D}\mbox{-a.s.}
\end{equation*}
From this and \eqref{phi} we get 
\begin{equation}
\label{eq3.cont1}
\hat u(X_{t}) \le a+\varepsilon,\quad t\in[0, \tau_D),\, P_{\bar x}^{c,D}\mbox{-a.s.}
\end{equation}
On the other hand, by the minorization condition $M(\bar c_D,
\psi_D^{\bar c_D}, \nu^{\bar c_D}_D)$, the assumptions made on
$\psi^{\bar c_D}_D$, $\nu^{\bar c_D}_D$, and \eqref{050309-20}, we conclude that for any $x\in D$
we have
\begin{align*}
\mathbb E^{c,D}_x\left[\int_0^{\tau_D} \mathbf{1}_{\{\hat u>a+2\varepsilon\}}(X_r)\,dr\right]&=R^{c,D}\mathbf{1}_{\{\hat u>a+2\varepsilon\}}(x)\ge R^{D}_{\bar c_D}\mathbf{1}_{\{\hat u>a+2\varepsilon\}}(x)
\\&\ge \psi_D^{\bar c_D}(x)\int_D \mathbf{1}_{\{\hat u>a+2\varepsilon\}}d \nu^{\bar c_D}_D>0.
\end{align*}
In the case of part 1) of the theorem  the positivity of the last
integral in the utmost right hand side follows
from  ${\rm supp}\,\nu^{\bar c_D} _D=D$ and  the continuity of $u$. In
the case of part 2) it is a consequence of  the
facts that $m( \hat u >a+2\varepsilon)>0$ and $m\ll \nu^{\bar c_D} _D$.
Hence, for every $x\in D$, we have
$$P_{x}^{c,D}\Big(\exists t\in[0,\tau_D):\,
\hat u(X_t)>a+2\varepsilon\Big)>0,\quad x\in D,
$$ 
which obviously contradicts (\ref{eq3.cont1}). Hence,
\eqref{012105-21} cannot be true and therefore  $\nu^{\bar c_D}_D(x:\,v^K(x)>0)>0$. This  ends
the proof of (\ref{eq.np1a}).\qed

\bigskip

\section{Minorization condition, irreducibility of the resolvent and strong maximum principle}
\label{mis.1}

Suppose that $m$ is a non-trivial, $\si$-finite Borel
measure in $D$. We say that the resolvent of $A$ is $m$-irreducible if
\begin{itemize}
\item[IR$_m$)]
for  any $f\in B_b^+(D)$ such that $\int_Df\, dm>0$  we have
\begin{equation}
\label{RD1}
R^D_1 f(x)>0\quad\mbox{for all } x\in D.
\end{equation}
\end{itemize}

 {
\begin{remark}
\label{rem.apx1}
Observe that for any $\si$-finite Borel measure $\nu$ on $D$ and non-negative Borel function $f$ on $D$ if 
$\int_D R^D_\alpha f\,d\nu>0$ for some $\alpha\ge 0$, then  $\int_D R^D_\alpha f\,d\nu>0$ for any $\alpha\ge 0$.
Indeed, it follows easily  that  the assumption made is equivalent to 
\begin{equation}
\label{eq.apx1}
\ell_1\Big(\big\{t\ge 0:\int_DP^D_tf\,d\nu>0\big\}\Big)>0.
\end{equation}
In particular if \eqref{RD1} holds, then $R^D_\al f\succ 0$ for all   $\al\ge 0$. 
\end{remark}}

 {
\begin{remark}
\label{rem.apx3}
Note that if the minorization condition $M(\alpha,\psi_D^\alpha,\nu^\alpha_D)$ holds for some $\alpha\ge 0$, strictly positive $\psi^\alpha_D$
and $\nu^\alpha_D$ satisfying $m\ll\nu^\alpha_D$, then the
$m$-irreducibility of $A$ follows.
\end{remark}}

%\begin{definition}
%\label{ref-meas}
%A $\sigma$-finite   Borel measure $\mu$ on $D$ is called a {\em reference measure}
%for the resolvent family $(R^D_\alpha)_{\alpha>0}$
%if $R^D_\alpha(x,\cdot)\ll\mu$ for any $x\in D$ and $\alpha>0$.
%\end{definition} 

Recall the definition of an irreducibility
measure for a non-negative kernel.
Given $\al>0$ we let $K_\al:D\times{\cal B}(D)\to[0,+\infty)$ be a
substochastic kernel defined by 
$$
K_\al(x,B):=\alpha R^D_\alpha1_B(x)\quad\text{ for
}\quad (x,B)\in D\times{\cal B}(D).
$$ 

\begin{definition}{(cf \cite[Section 2.2]{Num})}
\label{df011607-20}
Suppose that $\nu$ is  {a  non-trivial}  $\si$-finite Borel measure on $D$.
We say that
the kernel $K_\al$ is  $\nu$-irreducible and $\nu$ is then  called an {\em irreducibility measure} for $K_\al$, if
for any $B$ such that $\nu(B)>0$ and $x\in D$ we have $K^n_\al(x,B):=(\alpha
R_\alpha^D)^n1_B(x)>0$ for some $n\ge1$. 
 {The kernel is said to be irreducible if it is
  irreducible for some measure.}

A measure $\mu$ is
called a {\em maximal irreducibility measure} for $K_\al$ if it is an
irreducibility measure for the kernel and  any  irreducibility measure
$\nu$   satisfies $\nu\ll \mu$.
\end{definition}

%In some of our results we 
%shall need the following hypothesis.
% \begin{enumerate}
% \item[LSC)] Any $(P^D_t)$-excessive function is lower semicontinuous,
%   see Definition \ref{exces}.
% \end{enumerate}

%One easily checks that LSC) implies RM) (see e.g. \cite[page 197]{bg}).

It turns out that condition IR$_m$) suffices for the validity of
$M(\al, \psi_D^\al, \nu^\alpha_D)$ for any $\al\ge0$ with strictly positive function $\psi_D^\al$ and measure
$\nu^\alpha_D$ satisfying $m\ll \nu^\alpha_D$.

\begin{theorem}
\label{th3.hl.nuc1a}
Assume that IR$_m$)  holds and
\begin{equation}
\label{eq.equivmrm}
\mbox{for some (thus all) $\alpha\ge 0$  we have}\,\,\, mR^D_\alpha\ll
m.
\end{equation}
Then, for each
$\al\ge 0$ there exist an $(\al+1)$-excessive function
$\psi^\alpha_D:D\to(0,+\infty)$ and a measure $\nu^\alpha_D$
satisfying $m\ll \nu^\alpha_D$,
such that $M(\al, \psi_D^\al, \nu^\alpha_D)$ is in force.
\end{theorem} 
\begin{proof}
Suppose that $\al>0$.
Condition IR$_m$) implies that the
kernel 
$
K_\al(\cdot,\cdot)$ is {\em $m$-irreducible}.
 By \cite[part ii) of Proposition 2.4]{Num}
\begin{equation}
\label{eq4.mimf1}
\mu_\alpha(\cdot)= \sum_{n=1}^\infty \frac{1}{2^n}\int_D K^n_\al(x,\cdot)\,m(dx).
\end{equation}
is a maximal irreducibility measure for $K_\al$, and so $m\ll\mu_\alpha$.
Clearly $\mu_\alpha\sim \mu_\beta$ for any
$\alpha,\beta>0$, i.e. both measures have the same null sets. We
 let $\mu:=\mu_1$, so  $m\ll\mu$. Recall that
 \[
 (R^D_\alpha)^nf(x)=\int_0^\infty \frac{t^n}{n !}e^{-\alpha t}P^D_tf(x)\,dt,\quad x\in D,\, n\ge 1.
 \]
 From the above,  assumption $mR_\alpha^D\ll m$ and characterization \eqref{eq.apx1} we infer that  $m(R^D_\alpha)^n\ll m$
 for any $n\ge 1,\, \alpha\ge 0$. Thus, using \eqref{eq4.mimf1} we
 obtain $\mu\ll m$. Thus $\mu\sim m$ and  $m$ is a maximal irreducibility measure.

By \cite[Theorem 2.1]{Num} 
there exist a non-trivial with respect to $m$ (i.e. $\int_D\hat s_\alpha\,dm>0$) function  $\hat s_\alpha:D\rightarrow
[0,\infty)$ and a non-trivial  (i.e. $\hat\nu_\alpha(D)>0$) $\si$-finite Borel measure  $\hat\nu_\alpha$ on $D$,
 the so called {\em small function} and {\em small measure}, such that
  {\begin{equation}
\label{eq.sfsm1n}
\Big(R_{\alpha+1}^D\Big)^nf(x)\ge \hat s_\alpha(x)\int_D fd \hat\nu_{\al},\quad f\in B^+(D),\quad x\in D.
\end{equation}
We claim that in fact from \eqref{eq.sfsm1n} it follows that 
\begin{equation}
\label{eq.sfsm1}
R_{\alpha}^D  f(x)\ge R_{\alpha+1}^Df(x)\ge \hat s_\alpha(x)\int_D fd \hat\nu_{\al},\quad f\in B^+(D),\quad x\in D.
\end{equation}
Indeed, we obviously have $R_{\alpha+1}^Df\le R_{\alpha}^Df$ for  any $f\in B^+(D)$.
By the resolvent identity, see \eqref{res-id},  
\begin{equation}
\label{010310-20} R_{\alpha}^D
R_{\alpha+1}^D f=R_{\alpha+1}^D
R_{\alpha}^D f =R_{\alpha}^Df-
R_{\alpha+1}^D f\le R^D_{\alpha} f \quad \mbox{for any $f\in B^+(D)$}.
\end{equation} 
If $n\ge2$ in \eqref{eq.sfsm1n}, then we can write
\begin{equation*}
%\label{eq.sfsm1n1}
\begin{split}
&
R_{\alpha}^D \Big(R_{\alpha+1}^D\Big)^{n-2}f(x)\ge \Big(R_{\alpha}^D
R_{\alpha+1}^D\Big)\Big(R_{\alpha+1}^D\Big)^{n-2}f(x)\\
&
\ge \Big(R_{\alpha+1}^D\Big)^nf(x)\ge \hat s_\alpha(x)\int_D fd \hat\nu_{\al},\quad f\in B^+(D),\quad x\in D.
\end{split}
\end{equation*}
Iterating this procedure, we conclude \eqref{eq.sfsm1}.}
%\begin{equation}
%\label{eq.sfsm1}
%R_\alpha^Df(x)\ge \hat s_\alpha (x)\int_D fd \hat\nu_D^\alpha,\quad f\in B^+(D),\quad x\in D.
%\end{equation}
By \eqref{010310-20}, $ R_{\alpha}^D
R_{\alpha+1}^D f\le R^D_\alpha f$. Therefore, applying \eqref{eq.sfsm1}
for $R_{\alpha+1}^D f$ in place of $f$ we get
\begin{equation}
\label{eq.sfsm2}
R_\alpha ^Df(x)\ge   \hat s_\alpha (x)
\int_Df d\nu_D^\al,\quad f\in B^+(D),\quad x\in D,
\end{equation}
 where $\nu_D^\al(B):= \int_BR^D_{\alpha+1}1_Bd \hat\nu_\alpha$,
 $B\in{\cal B}(D)$. Since $\hat\nu_\alpha$
 is non-trivial, we get by IR$_m$) that $m\ll \nu_D^\al$.
Applying $R_{\alpha+1}^D$ to both sides of  \eqref{eq.sfsm2}
 we get \eqref{RDa} with $\psi_D^\al:=R_{\alpha+1}^D\hat s_\alpha $.  Note that $\psi_D^\al$ is an $(\al+1)$-excessive function with respect to $(P^D_t)$.
By hypothesis IR$_m$)  we conclude  that $\psi_D^\al$ is strictly positive on
$D$. 
% Inequality (\ref{RDa}) implies then that $\nu_D^\alpha$ is an irreducibility measure for $K_\al$, so
%$\nu_D^\alpha\ll \mu$ (by maximality of $\mu$). On the other hand, since $\mu$ is an irreducibility measure for $K$, if $\mu(B)>0$, then $R^D_{\alpha+1}\mathbf{1}_B(x)>0,\, x\in D$.
%Thus, $\nu^\alpha_D(B)=\int_DR^D_{\alpha+1}\mathbf{1}_B(x) \hat
%\nu_\alpha(dx)>0$. Consequently, $\nu_D^\alpha \sim\mu$ for each
%$\alpha>0$. In particular we also have $m\ll \nu^{\alpha}_D$. 
\end{proof}

%(\teka{ten wniosek pierwotnie  byl po Remark 4.7, 4.8})

From Theorems \ref{th3.hl.nuc1a} and \ref{th3.hl.nuc1} 
we immediately conclude the following.
\begin{corollary}
\label{cor011607-20}
Suppose  that the operator $A$ satisfies the assumptions of Theorem
  \ref{th3.hl.nuc1a}. Then, there exists a function
$\psi:D\rightarrow (0,\infty)$ such that for any $u\in\mathcal U^+_c$   non-constant $m$-a.e.   in $D$,
we can find $a>0$,  for which
\begin{equation}
\label{010104-20av1a}
\bar u_{D_S }-u(x)\ge a  \psi(x),\quad x\in  D.
\end{equation}
% Here   $\hat x$  satisfies \eqref{012810-20}.
\end{corollary}

{
\begin{remark}
Observe that IR$_m$) and \eqref{eq.equivmrm} together imply that $m\sim mR^D_\alpha$
for any $\alpha\ge 0$.
\end{remark}}

\begin{remark}
{Consider the following condition}
\begin{equation}
\label{bdd-a}
{\mbox{for some (thus all) $\alpha\ge 0$  we have}\,\,\, m\ll mR^D_\alpha.}
\end{equation}
Condition  \eqref{bdd-a}   holds in particular
when  the measure $m$ 
is   excessive.  Indeed, then $\alpha m R_\al^D\le m$ for all  $\al>0$ and, 
according to   \cite[Section 37b]{DM}, we can write   
\begin{equation}
\label{bdd-2}
\lim_{\al\to+\infty}\al\int_DR_{\al}^Dfdm=\int_Dfdm,\quad
\mbox{for any }f\in B_b^+(D),
\end{equation}
which implies that  {$m\ll mR^D_\alpha$}.
 {Obviously,  by the excessiveness of $m$, we have $m R^D_\alpha\ll m$.
Thus, in fact, the excessiveness of $m$ implies that $m\sim mR^D_\alpha$ for any $\alpha \ge 0$.}
\end{remark}
{
\begin{remark}
\label{rem.apx2}
It is easy to see that \eqref{bdd-a} also holds  if
\begin{equation}
\label{bdd}
\liminf_{t\to0+}\int_DP_t^Dfdm>0,\quad \mbox{for any }f\in
B_b^+(D) ,\quad\mbox{such that }\int_Dfdm>0.
\end{equation}
\end{remark}}

In our next result we    show a form of the converse to the 
result of Theorem \ref{th3.hl.nuc1a}. Namely, we have the following.

\begin{theorem}
\label{prop.apx1}
Assume that operator $A$, domain $D$ and an excessive measure $m$  are fixed.
Suppose furthermore that     the conclusion of Corollary
\ref{cor011607-20} holds for some  $c\ge 0$ and any $u\in
{\cal U}_c^+$.
%  there exists a Borel measurable function $\psi:D\to (0,\infty)$
% such that for any $u\in\mathcal U^+_c$ there exists $a>0$ for which
% \eqref{010104-20av1a} is in force.
% \[
% \bar u_{D_S}-u(x)\ge a\psi_D(x),\quad x\in D.
% \]
Then, the operator $A$ satisfies condition IR$_m$).
\end{theorem}

%\textcolor{red}{\bf DOTAD}

\begin{proof}
 Suppose that $f\in B_b^+(D)$ and $\int_D f\,dm>0$. By ET)  (see \cite[Proposition 2.2]{BCR}), there exists a strictly positive $h\in B_b(D)$
such that $R^Dh$ is bounded by a constant $m$-a.e. Set $u_n:=
-R^D(f\wedge (nh))$ for $n\ge1$. Directly from the definition one can
see that  $u_n$ is a weak subsolution to \eqref{eq3.1ab}, with the
given 
$c$.
If for some $n\ge1$ we have $\bar{u}_{n,D_S}<0$, then, clearly, $R^Df(x)>0,\, x\in D$.

Suppose, therefore that 
$\bar{u}_{n,D_S}=0$ for all $n\ge1$.  Note that then $R^D(f\wedge
(nh)) $ is
non-constant $m$-a.e. for some sufficiently large $n\ge 1$. Indeed, if otherwise we
would have 
$R^D(f\wedge (nh))\equiv 0$, $m$ a.e. for all $n\ge1$. By virtue of
Remark \ref{rem.apx1} then $R^D_\al(f\wedge (nh)) \equiv 0$  for
all $\al\ge0$ and $n\ge1$. From the fact that $m$ is excessive
 we     conclude, using \eqref{bdd-2}, that
\begin{equation*}
%\label{bdd-22}
0=\lim_{\al\to+\infty}\al\int_D R^D_\al(f\wedge (nh)) dm=\int_D
\big(f\wedge (nh)\big) dm,\quad n\ge1.
\end{equation*}
Letting $n\to+\infty$ leads to a contradiction with the assumption
that $f$ is non-trivial.

%\green{\em T.K.: I do not understand why  we need $h$?}

Since $u_n$ is
non-constant for some $n$ we can 
use estimate \eqref{010104-20av1a} (remember that
$\bar{u}_{n,D_S}=0$). There exists $a_n>0$ such that 
\[
R^D(f\wedge{nh})(x)=-u_n(x)\ge a_n \psi_D(x),\quad x\in D.
\]
Thus, again $R^Df(x)>0,\, x\in D$. By Remark \ref{rem.apx1}, we
conclude that the latter implies in fact \eqref{RD1}.
\end{proof}

%\green{\em T.K.: I am confused with this remark...How $c=1$ differs
 % from any other $c$.}

\begin{remark}
The conclusion of Theorem \ref{prop.apx1}   is still in force, if we assume
that  estimate \eqref{010104-20av1a} holds for any    $u\in
{\cal U}_1^+$ and, instead of the excessiveness of $m$,  we require  
the weaker
condition \eqref{bdd-a}.
 Indeed,
suppose that $f\in B_b^+(D)$ is such that $\int_Dfdm>0$. Then   $u:=-R^D_1f$
 is a weak subsolution to \eqref{eq3.1ab} with $c\equiv 1$. If
$\bar{u}_{D_S}<0$, then obviously $R^D_1f(x)>0$ in $D$. Thanks
to \eqref{010104-20av1a}  the same conclusion holds, if $R^D_1f$ is  not constant. If the latter does not hold, then the
only possibility for  $R^D_1f(x)$ to admit zero value is if it 
identically vanishes, but this would  contradict \eqref{bdd-a}.
\end{remark}

\begin{definition}[Strong maximum principle]
\label{SMPDa}
We say that the operator $-A+c$, with $A$ given by \eqref{A},
satisfies the strong maximum principle with respect to measure $m$ (SMP$_m$) if for any weak
  subsolution $u$ to
 \eqref{eq3.1ab} on $D$, for which $u(\hat x)=\bar u_{D_S }\ge 0$
at some $\hat x\in   D$ we have   $u\equiv \bar u_{D_S }$,  $m$-a.e. in $D$.
\end{definition}

\begin{theorem}
\label{prop012904-20v2}
 The 
hypothesis IR$_m$) implies 
the validity of the SMP$_m$ for the operator $-A+c$ with   any $c$
satisfying A3). 

Conversely, if we assume that \eqref{bdd-a}
and  the SMP$_m$    for   $-A+c$, with   $c\equiv 1$
  are in force, then IR$_m$) holds.
\end{theorem}
%The proof of the result is given in Section \ref{SMP}.
\begin{proof}
Assume condition IR$_m$). Let $u$ be a  weak
  subsolution  to
 \eqref{eq3.1ab} in $D$ such that $M:=\bar u_{D_S }\ge0$. Let
 $u_M:=M-u$. From  the fact that $M\ge 0$ and $u$ is a weak
 subsolution we conclude, via \eqref{subsol.g}, that $u_M$ satisfies
\[
u_M(x)\ge \mathbb E_x\left[e_c(t\wedge \tau_D)u_M(X_{t\wedge\tau_D})\right],\quad x\in D,\, t\ge 0.
\] 
Since $c$ is bounded, $e_c(t\wedge \tau_D)$ is strictly positive. 
% Moreover,
% by Lemma \ref{lm.lm.151120}, $u_M(X_{t\wedge \tau_D})\ge 0,\, t\ge 0$.
Suppose that there exists $x_0\in D$ such that $u(x_0)=M$. Then $u_M(x_0)=0.$
Therefore, by the above inequality and strict positivity of
$e_c(t\wedge \tau_D)$ we have
$\mathbb E_{x_0}u_M(X_{t\wedge \tau_D})=0$ for any  $t\ge 0$.
Thus,
\[
P^D_tu_M(x_0)\le \mathbb E_{x_0}u_M(X_{t\wedge \tau_D})=0,\quad t\ge 0
\]
and, as a result, $R^D_1u_M(x_0)=0$. Thanks to the irreducibility assumption IR$_m$) we have $u_M\equiv 0$, $m$-a.e.  on
$D$. Therefore $u\equiv M$, $m$-a.e.  on $D$ and the SMP$_m$ holds (cf Definition
\ref{SMPDa}).

Conversely, assume the SMP$_m$ and  \eqref{bdd-a}. We show that IR$_m$) holds.  We shall argue by contradiction. Let 
\begin{equation}
\label{011401-21}
f\in
B_b^+(D)\quad\mbox{ and }\quad  \int_Dfdm>0.
\end{equation}  
Then  $u:= -R^D_1f$  is a weak subsolution to (\ref{eq3.1a}), with
$c\equiv 1$.

 Assume that for some $x_0\in D$ we have $u(x_0)=0$.
 Thus,
$\bar u_{D_S} =0$ and
by the SMP$_m$ we have $u\equiv 0$,
$m$-a.e., which in turn implies that 
  $R^D_\al f\equiv 0$, $m$-a.e. for
all $\al\ge 0$. 
In light of \eqref{011401-21} this clearly
contradicts   \eqref{bdd-a}.
We conclude therefore that $R^D_1f\succ 0$ in  $D$, thus  IR$_m$) holds.
\end{proof}

{
\begin{remark}
If we assume  in Theorem \ref{prop012904-20v2} a {stronger condition} that $m$ is excessive instead of \eqref{bdd-a},
then, by  the argument used in the proof of Theorem
\ref{prop.apx1}, we conclude that the fact that SMP$_m$ holds for $-A+c$ with some  $c$ satisfying A3) implies IR$_m$) 
\end{remark}

%\textcolor{red}{\bf DOTAD}

\section{Relationship between the minorization condition and intrinsic ultracontractivity}
 
\label{disc2}

The concept of the intrinsic ultracontractivity has been
introduced for symmetric operators by Davies and Simon in \cite{DaS}
and later extended to the non-symmetric case by  Kim and Song in
\cite{KiS}, see also \cite{KnP}. 
We assume that $m$ is a  finite Borel measure on $D$ and sub-Markovian 
semigroups $\big(P^D_t\big)_{t\ge0}$, $\big(\hat P^D_t\big)_{t\ge0}$,
given by 
\begin{equation*}
\begin{split}
&P^D_tf(x):=\int_Dp_D(t,x,y)f(y)m(dy),\quad \\
&
\hat
P^D_tf(x):=\int_Dp_D(t,y,x)f(y)m(dy),\quad f\in L^\infty(D,m),
\end{split}
\end{equation*}
extend to   strongly continuous semigroups on $L^2(D,m)$.
Here   $(x,y)\mapsto p_D(t,x,
y)$, $(x,y)\in  D\times D$ are bounded, continuous and strictly
positive for each $t>0$. It follows
then, see Theorem V.6.6 on page 337 of \cite{schaefer}, see also
Proposition 2.3 of \cite{KiS}, that there
exist the principal eigenvalue $\la_D>0$ and strictly positive,
continuous   eigenfunctions
$\varphi_D,\, \hat{\varphi}_D:D\to(0,+\infty)$ belonging to $L^\infty(D,m)$ for both $P^D_t$ and
$\hat P^D_t$, i.e.  
\begin{equation*}
%\label{eigen1}
e^{-\lambda_D t}\varphi_D(x)=\int_Dp_D(t,x,y) \varphi_D(y)
m(dy),\quad t\ge 0,\,x\in D
\end{equation*}
and
\begin{equation*}
%\label{eigen2}
e^{-\lambda_D t}\hat{\varphi}_D(x)=\int_Dp_D(t,y,x) \hat{\varphi}_D(y)
m(dy),\quad t\ge 0,\,x\in D.
\end{equation*}
We normalize $\varphi_D$ and $\hat\varphi_D$ by letting  $\int_D\varphi_Ddm=\int_D\hat\varphi_Ddm=1$.

We say that $(P^D_t)_{t\ge 0}$ is {\em  intrinsically ultracontractive}, if 
for each $t>0$ there exists 
constant $c_t > 0$ such that
\begin{equation}
\label{011105-20}
p_D(t,x,y)\le c_t \varphi_D(x) \hat{\varphi}_D(y)\quad\mbox{for any }(x,y)\in D\times D.
\end{equation}
It is well known (see e.g.   \cite[Proposition 2.5]{KiS}), that condition
\eqref{011105-20} implies  the existence of a continuous function
$\tilde c:[0,+\infty)\to(0,+\infty)$ such that
\begin{equation*}
%\label{011105-20a}
p_D(t,x,y)\ge \tilde c(t) \varphi_D(x) \hat{\varphi}_D(y)\quad\mbox{for any }(x,y)\in D\times D.
\end{equation*}
This estimate immediately implies the following.  
\begin{theorem}
\label{prop5.1}
Assume that $(P^D_t)_{t\ge 0}$ is intrinsically ultracontractive. Then for any $\al>0$
condition $M(\al, \psi_D^\alpha, \nu^\alpha_D)$ holds, with 
$$
\psi_D^\alpha:=\varphi_D,\quad \nu_D^\alpha(dx):=\left(\int_0^{+\infty}\tilde c(s)e^{-\alpha s}ds\right)\hat\varphi_D(x)m(dx).
$$
\end{theorem}

\bigskip

%From Corollary \ref{cor010807-20} and Proposition \ref{prop5.1} we
%immediately conclude the following.
%\begin{corollary}
%\label{cor5.1}
%Under the assumptions of Proposition \ref{prop5.1} the conclusions made 
%in  parts 2) and 3) of Corollary \ref{cor010807-20} remain valid.
%\end{corollary}

There exists  a vast literature on the intrinsic ultracontractivity of
L\'evy type operators, see e.g. \cite{CW,KiS} and the references therein.
Recall   that if $(P_t)_{t\ge 0}$ is a symmetric L\'evy semigroup,
i.e.   for all $x\in D$ we have
$\mathbf{Q}(x)\equiv \mathbf{Q}(0)$, $ b\equiv 0$ and  $N(x,\,dy)\equiv
N(0,\,dy)$ is symmetric, then the full support 
of $N(0,\,dy)$ implies ultracontractivity of $(P^D_t)_{t\ge 0}$ for any bounded domain $D\subset \BR^d$ (see e.g. \cite{G}).
Observe also that ultracontractivity
of $(P^D_t)_{t\ge 0}$ implies that 
$
P^D_t1\sim \varphi_D$ on  $D$ for each $t>0$.
There are many results concerning the behavior of $P^D_t1$ (see
e.g. \cite{CS,GKK}) that state in particular that $P_t^D1\sim
\phi(\delta_D)$ on $D$ for some  function $\phi$, which can be
determined from the coefficients of  $A$ and the domain $D$. For
  this type of a result some sufficient regularity  of the domain   is
usually required, e.g. of $C^{1,1}$ class, see e.g. \cite{GKK}.
This in turn implies that
\begin{equation}
\label{sal.sal1}
\varphi_D\sim \phi(\delta_D)\quad\mbox{on}\quad D.
\end{equation}
For example, if $A=\Delta^{s/2}$ for some $s\in (0,2)$, i.e.
\begin{equation}
\label{Sal}
Au(x)=\lim_{h\searrow
  0}\int_{B^c(0,h)}[u(x+y)-u(x)]\frac{dy}{|y|^{s+d}},\quad u\in
C^2(\bbR^d)\cap C_b(\bbR^d),\,x\in\BR^d,
\end{equation}
then $\phi(t)=t^{s/2},\, t\ge 0$ for a sufficiently regular domain $D$, see e.g. \cite{Kulczycki}.

\bigskip

Therefore, directly from Theorems  \ref{th3.hl.nuc1} and \ref{prop5.1},   we
conclude the following version of the Hopf lemma.
\begin{corollary}
\label{cor010907-20}
Suppose that   \eqref{sal.sal1} holds and the assumptions of Theorem
\ref{prop5.1} are in force. 
Let $u$ be a non-constant $\ell_d$-a.e.  weak subsolution to (\ref{eq3.1ab}). 
Then  there exists $a>0$ such that for any $\hat x\in   \partial D$    with  $u(\hat
x)=\bar u_{D_S }\ge0$  we have 
\begin{equation}
\label{010104-20av1cor}
\mathop{\liminf_{x\to \hat x}}_{x\in D}\frac{u(\hat x)-u(x)}{\phi(\delta_D(x))}\ge a.
\end{equation}
\end{corollary}

\begin{remark}
In particular, if  $D$ satisfies the interior ball condition and ${\bf
  n}(\hat x)$ is the normal outward vector  at $\hat x$, then for sufficiently small $h>0$,
\[
\frac{u(\hat x)-u(\hat x-h {\bf n} (\hat x))}{\phi(h)}\sim\frac{u(\hat x)-u(\hat x-h {\bf n} (\hat x))}{\phi(\delta_D(\hat x-h {\bf n} (\hat x)))}.
\]
So, by Corollary \ref{cor010907-20}, the $\phi$-normal derivative at a
maximal point $\hat x$, defined as
\[
\partial^\phi_{\bf n}u(\hat x):= \liminf_{h\to 0^+}\frac{u(\hat x)-u(\hat x-h {\bf n} (\hat x))}{\phi(h)},
\]
is strictly positive.%, if $c$ is non-negative, or bounded from below by $a u(\hat
%x)$, if in addition $\int_Dc(x)dx>0$.
\end{remark}

\begin{remark}
\label{rm5.4}
  Estimate   \eqref{010104-20av1cor}
is an analogue of the classical Hopf lemma.  As we have mentioned in
the foregoing, it does require,
as in the classical case,  regularity of the domain. Recall that the classical
Hopf lemma for the Laplace operator does not hold in general,  even for $C^{1}$-regular
domains. 
However, in many applications 
estimates of the type   \eqref{010104-20av1} with $\psi^{\bar c_D}_D=\varphi_D$ are quite sufficient. 
In such a situation there is no need for a  strong regularity
assumption on $D$, as in the case of estimates 
\eqref{010104-20av1cor}.
For many    operators $A$ regularity
assumptions on $D$ are only  required  to characterize the behaviour
of the principal eigenfunction near  the boundary of the domain.
More specifically, if $A=\Delta^{s/2}$ for some $s\in (0,2)$, then for
any bounded domain we have  \eqref{010104-20av1} with $\psi^{\bar c_D}_D=\varphi_D$, since
it is well known (see e.g. \cite[Theorem 1]{CW}) that $(P^D_t)_{t\ge 0}$ is intrinsic ultracontractive, regardless of the regularity of $D$.
On the other hand even for  Lipschitz domains $D$ it is not true that
$\varphi_D\sim \delta_D^{s/2}$. So, we cannot expect \eqref{010104-20av1cor} to hold for a non-regular $D$.
\end{remark}

%\subsection{Minorization by a non-negative eigenfunction and
  %  $\varphi_D$-Hopf lemmas}

%Consider the case when $\psi^\alpha_D$   in \eqref{RDa} equals
%$\varphi_D$ - a non-negative eigenfunction for $A$. More precisely, we assume
%the following version of the minorization condition.
%\begin{definition}[Minorization
%condition $M'(\al, \nu^\alpha_D)$]
%\label{minorization1}
%{Suppose
%that $\al\ge0$ and  $\nu^\alpha_D$ is a non-trivial 
%Borel measure  on $D$. %\textcolor{red}{with full support.} 
%We say that the minorization
%condition $M'(\al, \nu^\alpha_D)$ is satisfied if 
%Eig) holds and
%\begin{equation}
%\label{RDap}
%R^D_\alpha f(x)\ge \varphi_D(x) \int_Dfd\nu_D^\alpha,\quad x\in
%D,\,f\in B_b^+(D).
%\end{equation}}
%\end{definition}

\section{
Minorization condition with bottom function as an eigenfunction
   and ergodic properties of the canonical
process}
\label{disc}
 {In the previous section, we have  proved that intrinsic ultracontractivity
of the semigroup $(P^D_t)$ implies $M(\alpha,\psi_D^\alpha,\nu_D^\alpha)$ with $\psi^\alpha_D=\varphi_D$,}
where $\varphi_D$ is the respective principal eigenfunction.
Here we discuss the relationship between the
minorization condition   
and  ergodic properties of the canonical process.    It turns out, see Theorem
\ref{eq22.th1} below, that under the hypotheses of irreducibility (condition IR$_m$)) and
 finiteness of the exit time of the canonical process (condition ET)),
the validity of $M(\alpha,\varphi_D, \nu^{\alpha}_D)$ is equivalent with the property of uniform
ergodicity of the resolvent family corresponding to the
$\varphi_D$-process obtained from the canonical
process via the $\varphi_D$-Doob transform, see Definition \ref{df020211-20}.
We also discuss a relation between uniform conditional ergodicity of the canonical process and the
minorization condition. Recall that $m$ is some non-trivial, $\si$-finite Borel measure on $D$.

%\subsubsection{Non-negative eigenfunction and eigenvalue}
%Using the transition semigroup we can define  a non-negative
%eigenvalue and  eigenfunction for the operator   $A$. 
\begin{definition}
A function $\varphi_D$ and $\lambda_D$  are called   a {\em non-negative
eigenfunction} and {\em   eigenvalue} for $A$, see
e.g. \cite{BNV}, if $\varphi_D:D\to[0,+\infty)$,  $\lambda_D\ge 0$,
\begin{equation}
\label{eigen}
e^{-\lambda_D t}\varphi_D(x)=P^D_t\varphi_D(x),\quad t\ge 0,\,x\in D
\end{equation}
and  $\int_D \varphi_Ddm>0$.
\end{definition} 

In some of our results we shall assume the condition.
\begin{itemize}
\item[Eig)]
The operator $A$ possesses a non-negative
eigenfunction $\varphi_D$  and eigenvalue  $\lambda_D$.
\end{itemize}

 {Observe that $\varphi_D$ is an excessive function with respect to $(P^D_t)$.}

\begin{remark}
Obviously hypothesis IR$_m$) implies that  $\varphi_D(x)>0$ for all
$x\in D$.   Under some additional
assumptions $\la_D$
could be  a unique eigenvalue with a positive eigenfunction. If,   it
is indeed simple, then   $\varphi_D$ can be determined, up
to a multiplicative constant.  In such a situation the pair $(\la_D,
\varphi_D)$  is  
called  the {\em principal pair} while $\la_D$, $\varphi_D$ are  
referred to as
the principal eigenvalue and eigenfunction for $A$, respectively.
 {A discussion on  the validity of
Eig) is carried out in  Section \ref{sub.sub.uewo4.2}  below. }
% In fact $\varphi_D$ is
% determined uniquely,  up to a multiplicative constant,  by \eqref{eigen}  among positive functions.
\end{remark}

Throughout this section we suppose the irreducibility condition $IR_m)$ and the existence of a
non-negative eigenfunction $\varphi_D$ (condition Eig)). Then, as we
have already mentioned, 
$\varphi_D\succ 0$ in $ D$.  Consider    $(P^{D, \varphi_D}_t)$ - the  transition probability
semigroup of  the   conservative $\varphi_D$-process (see \cite[Chapter
11]{chung}), defined by 
\begin{equation}
\label{Ptphi}
P^{D, \varphi_D}_t f(x) :=e^{\la_D t}\frac{P^D_t(f
  \varphi_D)(x)}{\varphi_D(x)},\quad f\in B_b(D),\,x\in D.
\end{equation}
Let $(R^{D, \varphi_D}_\alpha)_{\al>0}$ be the respective resolvent family.

\subsection{Uniform ergodicity of a Markov
  process and its resolvent}

Let us recall  the notion of the uniform ergodicity of a Markov
  process, see 
  \cite[p. 1675]{tweedie}.
\begin{definition} 
Suppose that $(X_t)_{t\ge0}$ is an $E$-valued Markov process with the transition
probabilities $P_t(x,\cdot)$, $t\ge0$, $x\in E$. We say that its
transition semigroup $(P_t)_{t\ge0}$ is {\em uniformly ergodic} if there exists 
 $\Pi_{\rm inv}$ - an invariant
Borel 
probability measure on $E$, constants  
$C>0$ and
$\rho\in(0,1)$ such that, see \eqref{TV},
$$
\sup_{x\in E}\|P_t(x,\cdot)-\Pi_{\rm inv}\|_{\rm TV}\le C \rho^t,\quad
t\ge0.
$$
An analogous definition can be formulated for a discrete time Markov chain.
\end{definition}

Given $\al>0$ and a Markov process $(X_t)_{t\ge0}$, with the
respective resolvent family $(R_{\al})_{\al>0}$, consider the transition probability kernel
\begin{equation}
\label{kalb}
K_\al(x,B):=\al R_\alpha1_B(x),\quad x\in E,\,B\in
{\cal B}(E).
\end{equation}
We suppose that a  Markov chain $(\tilde X_n^{(\al)})_{n\ge0}$ has
the transition probability given by \eqref{kalb}.

\begin{definition} 
\label{df020211-20}
We say that the resolvent family $(R_{\al})_{\al>0}$ is uniformly ergodic if the
  chain $(\tilde X_n^{(\al)})_{n\ge0} $ is uniformly ergodic
  for any $\al>0$.
\end{definition}

\subsection{Minorization condition and  uniform ergodicity of the
  resolvent}

It turns out that, if $(X_t)_{t\ge0}$ is irreducible and $\varphi_D$
is bounded, then  the
minorization condition $M(\alpha,\varphi_D,\nu^\alpha_D)$ is equivalent
with the uniform ergodicity of the resolvent of the
$\varphi_D$-process. 
% This
% characterization of the minorization  is the
% content of our next result. 
\begin{theorem}
\label{eq22.th1}
Assume that IR$_m$) and Eig) hold,  the eigenfunction $\varphi_D$ of
Eig) is   bounded,  and $mR_\alpha^D \ll m$ for some $\alpha\ge 0$.  Then, condition $M(\al,\varphi_D, \nu^\alpha_D)$ is satisfied with non-trivial measure $\nu^\alpha_D$ for any
$\alpha\ge 0$  if and only if $(R^{D,\varphi_D}_{\al})_{\al>0}$ is
uniformly  ergodic. Moreover,  if $(P^{D,\varphi_D}_t)_{t\ge 0}$ is uniformly ergodic,
then $M(\al,\varphi_D,\nu^\alpha_D)$ holds for any $\alpha\ge 0$ with $m\ll\nu^{\alpha}_D$ .
\end{theorem}
%The result is shown in Section \ref{sec9.11}.
\begin{proof}

Suppose first that the $\varphi_D$-process $(X_t)_{t\ge 0}$ is
uniformly ergodic with $\Pi$ the unique invariant probability measure. Then for every $f\in B^+_b(D)$ there exists $t_f>0$ such that
\[
P^{D,\varphi_D}_tf\ge \frac12\int_D f\,d\Pi,\quad t\ge t_{f}.
\]
From the definition \eqref{Ptphi} we conclude
\[
P^{D}_tf\ge \frac{\varphi_D e^{-\lambda_D t}}{2\|\varphi_D\|_\infty}\int_Df\,d\Pi,\quad t\ge t_{f}.
\]
If $w$ is a bounded, $\alpha$-excessive function, with respect to
$(P^D_t)_{t\ge 0}$, then there exists $t_w>0$ such that
\[
w\ge e^{-\alpha t}P^D_tw\ge  \frac{\varphi_D e^{-(\lambda_D+\alpha) t}}{2\|\varphi_D\|_\infty}\int_Dw\,d\Pi,
\quad t\ge t_w.
\]
Therefore, for any $\alpha$-excessive non-trivial $m$-a.e. function $w$ one can
find a constant $c_w>0$ such that 
\begin{equation}
\label{w-cw}
w\ge c_w\varphi_D.
\end{equation}
Using Theorem \ref{th3.hl.nuc1a}, we   conclude the existence of an
$\al+1$-excessive 
function $\psi_D^\al$  and measure $\nu^{\al}_D$, that dominates the
Lebesgue measure $m$ such that \eqref{RDa} holds.
Combining this fact with the argument presented in the foregoing we conclude
that there exists $c>0$ such that $\psi_D^\al\ge c\varphi_D$, and $M(\al,\varphi_D,\nu^\alpha_D)$ follows.

Suppose now that $(R_{\beta}^{D,\varphi_D})_{\beta>0}$ is uniformly
 ergodic. Fix any $\al_0>\la_D+1$. We establish \eqref{RDa} with $\psi_D^\alpha=\varphi_D$ for any $\al\in [0, \al_0-\la_D-1)$. This obviously
 would imply the validity of the condition for any $\al\ge0$.

Given $\al\in [0, \al_0-\la_D-1)$ choose  $\psi_D^\al$ and
$ \nu^{\al}_D$  satisfying the conclusion of Theorem
\ref{th3.hl.nuc1a}. In particular, $\psi_D^\al$ is $\al+1$-excessive and  $m\ll \nu^{\al}_D$.
Observe that 
% \marginpar{{Cor.}}
 \begin{equation}
\label{RDal}
 R^{D,\varphi_D}_\beta f=\varphi_D^{-1}R^D_{\beta-\lambda_D}(f \varphi_D),\quad \beta\ge \la_D,\, f\in B^+(D).
 \end{equation}
 Using the uniform ergodicity of the Markov chain corresponding to the
 kernel $\al_0R_{\al_0}^{D,\varphi_D}$ and
\eqref{RDal} for $\beta=\al_0$ we conclude, as in the foregoing, that  there exists  $\tilde
\Pi^{(\al_0)}$ - a probability measure - such that for any  $f\in
B^+_b(D)$ one can find $n_f>0$ for which
\[
\left(R^D_{\alpha_0-\lambda_D}\right)^n f \ge \frac{\varphi_D}{2 \al_0^n\|\varphi_D\|_\infty}\int_D f\, d \tilde
\Pi^{(\al_0)},\quad n\ge n_{f}.
\]
Hence, if $w$ is any non-trivial $m$-a.e.  $\al+1$-excessive function  (w.r.t. $(P_t^D)$) we get that
there exists $n_w>0$ such that
\[
w \ge \frac{\varphi_D}{2\|\varphi_D\|_\infty}  \left(1-\frac{\lambda_D+\al+1}{\al_0}\right)^n\int_D w\, d \tilde
\Pi^{(\al_0)},\quad n\ge n_{w}
\]
and \eqref{w-cw} is in force. Thus, $\psi_D^\al\ge c\varphi_D$ for
some constant $c>0$.   Theorem
\ref{th3.hl.nuc1a} allows us to conclude then  $M(\alpha,\varphi_D,\nu^\alpha_D)$.

%%%%%%%COMPLETE ARGUMENT DO  NOT REMOVE%%%%%%%%%%%%%%%
% Suppose now that $(R_{\al}^{D,\varphi_D})_{\al>0}$ is uniformly
% ergodic. Given  $\al_0>\la_D$ there exists $n_f$ and $\tilde
% \Pi^{(\al_0)}$ - a probability measure - such that for any  $f\in B^+_b(D)$ there exists $n_f>0$, for which
% \[
% \left(\al_0 R^{D,\varphi_D}_{\al_0}\right)^n f\ge \frac12\int_D f\,d \tilde
% \Pi^{(\al_0)},\quad n\ge n_{f}.
% \]
% Observe that
% \marginpar{{Cor.}}
% \[
% R^{D,\varphi_D}_\alpha f=\varphi_D^{-1}R^D_{\alpha-\lambda_D}(f \varphi_D),\quad {\alpha\ge \la_D},\, f\in B^+(D).
% \]
% Therefore
% \[
% \left(R^D_{\alpha_0-\lambda_D}\right)^n f \ge \frac{\varphi_D}{2 \al_0^n\|\varphi_D\|_\infty}\int_D f\, d \tilde
% \Pi^{(\al_0)},\quad n\ge n_{f}.
% \]
% If $w$ is $\al$-excessive and $0<\al<\alpha_0-\lambda_D$ we get
% \[
% w \ge \frac{\varphi_D}{2 [\al_0 \left(\alpha_0-\lambda_D-\al\right)]^n\|\varphi_D\|_\infty}\int_D w\, d \tilde
% \Pi^{(\al_0)},\quad n\ge n_{f}
% \]
% and \eqref{w-cw} is in force. Another application of Theorem
% \ref{th3.hl.nuc1a} allows us to conclude the
% estimate \eqref{RDap}.
%%%%%%%THE END OF THE ARGUMENT%%%%%%%%%%%%%%%

Suppose now that $M(\al,\varphi_D, \nu^\alpha_D)$
holds. Then, by virtue of \eqref{RDal}, the measure  
$d\tilde\nu^\alpha_D:=\varphi_D d\nu^\alpha_D$ satisfies
\begin{equation*}
% \label{tRDa}
R^{D, \varphi_D}_{\alpha}f\ge   R^{D, \varphi_D}_{\alpha+\la_D}f\ge \int_D fd\tilde\nu^\alpha_D,\quad x\in D,\, f\in B_b^+(D),\,\al>0.
\end{equation*}
% \begin{equation}
%  \label{tRDa}
% \begin{split}
% &
% R^{D, \varphi_D}_{\alpha+\la_D}
% f(x):=\int_0^{\infty}e^{-(\alpha+\la_D) s}P^{D, \varphi_D}_s f(x)
% ds\\
% &
% = \frac{R^{D}_{\al} (f
%   \varphi_D)(x)}{\varphi_D(x)}\ge \int_D fd\tilde\nu^\alpha_D,\quad x\in D,\, f\in B_b^+(D).
% \end{split}
% \end{equation}
%} 
% This condition appears   in the definition of a pettite set, see e.g. \cite[Section
% 4]{MT2}. It guarantees, via
Using
 the Dobrushin theorem, see \cite[Theorem
2.3.1, p. 33]{kulik}, we conclude the uniform   ergodicity  in
the total variation norm   of the Markov chain $(\tilde
X_n^{(\al)})_{n\ge0}$ and the proof of Theorem \ref{eq22.th1} is
therefore concluded.
\end{proof}

 %\textcolor{red}{\bf DOTAD}

\subsection{Uniform conditional ergodicity and minorization condition
  $M'(\al, \nu^\alpha_D)$}

\begin{definition} 
The canonical process $(X_t)_{t\ge0}$ is called {\em uniformly conditionally
ergodic,}  cf
\cite[Section 2.2]{KnP}, if
there exists a (unique) probability Borel measure $\Pi$ on $D$, called {\em quasi-stationary distribution},
such that
\begin{equation}
\label{eq.uce1}
\lim _{t\rightarrow +\infty}\sup_{B\in\BB(D),x\in
  D}\Big|P_x\big(X_t\in B\big| \,t<\tau_D\big)-\Pi(B)\Big|=0.
\end{equation}
\end{definition}

It turns out that the minorization condition $M'(\al, \nu^\alpha_D)$ is 
a consequence of  the  uniform conditional
ergodicity of $(X_t)_{t\ge 0}$. This is a consequence of  Theorem
\ref{eq22.th1} and the following result.
\begin{proposition}
\label{prop010807-20}
Suppose that the assumptions of Theorem \ref{eq22.th1} are in
  force. If the canonical process $(X_t)_{t\ge 0}$ is uniformly
conditionally ergodic, then  $(P^{D,\varphi_D}_t)_{t\ge 0}$ is uniformly ergodic.
\end{proposition}
%The proof of the proposition is shown in Section \ref{sec11a}.
\begin{proof}
Condition (\ref{eq.uce1}) is   equivalent with
\begin{equation}
\label{eq.uce1a}
\lim _{t\rightarrow \infty}\sup_{B\in\BB(D),x\in
  D}\Big|\frac{P^D_t(x,B)}{P^D_t(x,D)}-\Pi(B)\Big|=0.
\end{equation}
Here $P^D_t(x,B):=P^D_t1_B(x)$.
From (\ref{eq.uce1a}) we obtain
\[
\lim_{t\rightarrow \infty}\sup_{x\in D}\sup_{ 
  \|f\|_\infty\le 1}\Big|\frac{ P^D_t(\varphi_D f)(x)}{P^D_t1(x)}-\int_Df\varphi_Dd\Pi\Big|=0.
\]
Equivalently,
\[
\lim_{t\rightarrow \infty}\sup_{x\in D}\sup_ {\|f\|_\infty\le 1}\Big|\frac{P^D_t\varphi_D(x)}{P^D_t1(x)}P^{D,\varphi_D}_tf(x)-\int_Df\varphi_Dd\Pi\Big|=0.
\]
Condition (\ref{eq.uce1}) implies that $P^D_t\varphi_D/P^D_t1$
converges uniformly to $a:=\int_D\varphi_D\,d\Pi$, as $t\to+\infty$. Therefore,
\[
\lim_{t\rightarrow \infty}\sup_{x\in D}\sup_{  \|f\|_\infty\le 1}\Big|P^{D,\varphi_D}_tf(x)-a^{-1}\int_Df\varphi_Dd\Pi\Big|=0
\]
and the conclusion of the proposition follows.
\end{proof}

%\textcolor{red}{\bf DOTAD}

\subsection{$\varphi_D$-Hopf lemmas and uniform ergodicity of the resolvent}

 We close this section with an interesting    corollary
  to the results presented above. It states esentially that
  the Hopf lemma with  the bottom function  $\varphi_D$ (a positive eigenfunction)
is equivalent with the uniform  ergodicity  of the resolvent family
corresponding to the $\varphi_D$-Doob transform of the canonical process.

\begin{definition}
\label{df010211-20}
Assume  that condition  Eig) is in force.
We say that a {\em $\varphi_D$-Hopf lemma} holds  for subsolutions to \eqref{eq3.1ab}
if  for any  $u\in\mathcal U^+_c$, that  is   non-constant $m$-a.e. in $D$, there exists $a>0$ such that 
\begin{equation}
\label{0101adcg}
\bar u_{D}-u(x)\ge a  \varphi_D(x),\quad x\in D.
\end{equation}
%If   $a= \bar u_{D_S }$,  then  we say that 
%a {\em quantitative $\varphi_D$-Hopf lemma} holds.
\end{definition}

\begin{proposition}
\label{prop26.1}
Assume  that  $mR_\alpha^D\ll m$ for some $\alpha\ge 0$, and  conditions  IR$_m$) and Eig) hold. Suppose  that for any $\al\ge0$ and for
any non-trivial function $f\in B^+(D)$ there exists $c(f,\al)>0$
such that 
\begin{equation}
\label{eq25.1}
R^D_\al f\ge c(f,\al)\varphi_D,\quad\mbox{in}\quad D.
\end{equation}
Then, $M(\alpha,\varphi_D,\nu^\alpha_D)$ holds for any $\al\ge0$, with $m\ll\nu^\alpha_D$.
\end{proposition}
\begin{proof}
Since IR$_m$) and $mR^D_\alpha\ll m$ are
assumed, by Theorem   \ref{th3.hl.nuc1a}, there exist
  a measure $\nu^\al_D$, satisfying
$m\ll\nu^\al_D$, and a strictly positive, $\al+1$-excessive  function $\psi^\al_D$
such that
estimate \eqref{RDa} holds. 
   Thus, by \eqref{eq25.1} applied for $\al+2$ and  $f=\psi^\al_D$, we have 
%$M'(\al,\nu^\al_D)$.}
% In addition, for any $\al>0$ we also have 
 \begin{equation}
\label{013110-20}
 \psi^\al_D\ge  R_{\al+2}^D\psi^\al_D\ge c(\psi^\al_D,\al+2)\varphi_D.
 \end{equation}
 The conclusion of the proposition follows from an application of
 \eqref{RDa} and then \eqref{013110-20}.
\end{proof}

% Let  $c(\cdot)$ be a Borel measurable function satisfying
%  A3). 
 
 \begin{theorem}
\label{thm010211-20}
Suppose that the assumptions of Theorem \ref{eq22.th1} are in force. The  $\varphi_D$-Hopf lemma, as
 formulated in Definition \ref{df010211-20}, holds  for  any $c(\cdot)$
 satisfying the hypothesis A3) if and only if the resolvent family $(R^{D,\varphi_D}_\al)_{\al> 0}$ is uniformly ergodic. 
\end{theorem}
\begin{proof}
Assume the validity of the $\varphi_D$-Hopf lemma. Suppose that
$\al\ge0$ and 
that   $f\in
B^+_b(D)$  satisfies $\int_D fdm>0$. Then 
  $u:=-R^D_\al f$
is a weak subsolution to $-Av+\al v=0$ %\textcolor{red}{From the definition
%of the resolvent  $u(x)=-R^D_\al f(x)=0$, $x\in \partial_SD$.} 
and, as a result,  
$u\in {\cal U}_\al$. Clearly $\bar u_{D_S}\le 0$.  If $\bar u_{D_S}<0$, then obviously
\eqref{eq25.1} holds (as $\varphi_D$ is assumed to be bounded).
Suppose therefore that $\bar u_{D_S}= 0$.  It cannot be constant, as
then we would have $\int_DR^D_\al fdm=0$, which in turn, by the
assumption that $mR_\al\ll m$, would imply that  $\int_D fdm=0$. The latter
contradicts the fact that $f$ is non-trivial $m$ a.e. We conclude from \eqref{0101adcg}
that there exists $a>0$ such that
$$
R^D_\al f(x) =  R^D_\al f(x)  +\bar u_{D_S}=\bar u_{D_S}-u(x)\ge a  \varphi_D(x).
$$
Thus  \eqref{eq25.1} follows again.  From
Proposition \ref{prop26.1} and Theorem \ref{eq22.th1}  we infer therefore the
uniform ergodicity of $(R^{D,\varphi_D}_\al)_{\al> 0}$. 

Conversely, suppose that  the resolvent family  is uniformly
ergodic. Then, by Theorem \ref{eq22.th1} for any $\al\ge0$  condition $M(\al,\varphi_D,
\nu^\alpha_D)$ holds with non-trivial  $\nu^\alpha_D$ satisfying $m\ll \nu^\alpha_D$. The
$\varphi_D$-Hopf lemma is then a direct consequence of Theorem  \ref{th3.hl.nuc1}.
\end{proof}

%\textcolor{red}{\bf DOTAD}

\section{Quantitative Hopf lemma}

\label{sec3-2810-20}
Recall that $\mathcal U_c(g)$  denotes the set
of all   weak  
subsolutions to \eqref{eq3.1a} and $\mathcal U_c^+(g)$ consists of those
$u\in \mathcal U_c(g)$, for which   $\bar u_{D_S }\ge0$. Given a function
$f:D\to\bbR$ we let  $f^-=\max\{-f,0\}$.  Throughout
this section we assume that $\bar g_D\le 0$, therefore $g^-=-g$.

{A starting point of our discussion  is the following simple result.}
\begin{proposition}
\label{prop010705-20}
Suppose that   $u\in \mathcal U_c(g)$. Then, cf \eqref{wcD},
\begin{equation}
\label{010705-20}
\bar u_{D_S }-u(x)\ge \bar u_{D_S } w_{c,D}(x) ,\quad x\in D.
\end{equation}
\end{proposition}
Note that  $w_{c,D}$ is  entirely  determined by the operator $A$ and function
$c$. Clearly,  \eqref{010705-20} is non-trivial as long as  function
$w_{c,D}$ does not vanish.
The latter is guaranteed e.g.  by the hypothesis IR$_m$)  and A3') as can be seen from
Proposition \ref{prop011009-20} below.  We see that in case $\bar u_D>0$ and $c$ is in some sense non-trivial, then 
the Hopf lemma holds with $a=\bar u_D$ and the bottom function
$w_{c,D}$.   {A result of this kind, where the dependence of the constant $a>0$ on  $\bar u_D$ is explicit, e.g. $a=C(\bar u_D+1)$, and  the constant $C>0$   depends only  
  on the coefficients $c, g$,
or operator $A$  shall be called a  {\em quantitative Hopf lemma}}.

\subsubsection*{ Proof of Proposition \ref{prop010705-20}.}
%Proof of the proposition is presented in Section \ref{sec5.1b}. 

With no loss of generality we may and shall assume that $\bar u_{ D_S }\ge
0$. Otherwise the estimate  \eqref{010705-20} is trivial.
By the definition of a weak subsolution and Lemma \ref{lm.lm.151120}, we can write
\begin{align*}
&\bar u_{ D_S }-u(x)\ge \bar u_{ D_S }-\E_x\left[e_c(\tau_D\wedge
                t)u(X_{\tau_D\wedge t})\right]\\
&\ge 
\bar u_{ D_S }-\bar u_{ D_S }\, \E_x\left[e_c(\tau_D\wedge
                t)\right]\ge \bar u_{ D_S }\left\{1- \E_x\left[e_c(\tau_D\wedge
                t)\right]\right\}\\
&\ge  w_{c,D}(x) \bar u_{ D_S },\quad x\in D
\end{align*}
and estimate \eqref{010705-20} follows.\qed

In what follows we formulate a number of refinements of Proposition \ref{prop010705-20}.

\subsection{Quantitative Hopf lemmas}
 We start with the following.

\begin{theorem}[Quantitative $\varphi_D$-Hopf lemma]
\label{pee}
Suppose that Eig) holds.
Then, $($see \eqref{c-c}$)$ for any  $u\in\mathcal U^+_c(g)$,
\begin{equation}
\label{012804-20}
\bar u_{D_S }-u(x)\ge \frac{\varphi_D(x)}{2e
  \|\varphi_D\|_\infty }\left(\frac{\underline
  c_D \bar u_{D_S }}{\la_D+\underline c_D}+\frac{\underline{ g}^-_D}{\la_D+\bar c_D}
\right),\quad x\in D.
\end{equation}
\end{theorem}
%The result is proved  in Section \ref{sec7.2.1}.\label{sec7.2.1}
\begin{proof}
By the definition of a weak subsolution, Lemma \ref{lm.lm.151120} and \eqref{eq.mu.main},
\begin{equation}
\label{052904-20b}
\bar u_{ D_S }-u(x)\ge \big(1-\E_xe_c(\tau_D\wedge t)\big) \bar u_{ D_S }-\mathbb
E_x\left[\int_0^{\tau_D\wedge t}e_c(s)g(X_s)\,ds\right],\quad t>0.
\end{equation}
Observe that (here $\underline c_D=\inf_Dc$)
\begin{equation}
\label{052904-20}
1-\E_xe_c(\tau_D\wedge t)\ge (1-e^{-{\underline c}_D t})P_x(\tau_D>t) .
\end{equation}
From \eqref{eigen} we have
\begin{equation}
\label{062904-20}
e^{-\lambda_D t}\varphi_D(x)=\E_x\left[\varphi_D(X_t),\,\tau_D>t\right]\le \|\varphi_D\|_\infty P_x(\tau_D>t).
\end{equation}
Substituting into \eqref{052904-20} the lower bound on
$P_x(\tau_D>t)$ obtained from \eqref{062904-20}
and maximizing over $t>0$ we conclude that
\begin{equation}
\label{072904-20}
\begin{split}
&\bar u_{ D_S }-u(x)\ge \frac{\bar u_{ D_S }
  \varphi_D(x)}{\|\varphi_D\|_\infty}\cdot \frac{{\underline c}_D /\la_D}{(1+{\underline c}_D/\la_D)^{1 +\la_D/{\underline c}_D}}\\
&
\ge \frac{\bar u_{ D_S }
  \varphi_D(x)}{e\|\varphi_D\|_\infty}\cdot\frac{{\underline c}_D}{\la_D+{\underline c}_D}.
\end{split}
\end{equation}
In the last inequality we use an elementary bound $(1+s)^{1/s}\le e$
valid for
$s>0$. We also have  (here, as we recall, $\bar c_D=\sup_Dc$ and ${\underline g}^-_D=\inf_Dg^-$)
\begin{equation*}
%\label{052904-20c}
\begin{split}
&-\mathbb
E_x\left[\int_0^{\tau_D\wedge t}e_c(r)g(X_r)\,dr\right]\ge\mathbb
  E_x\left[\int_0^{\tau_D\wedge t}e^{-\bar c_D r}dr\right] \underline{ g}^-_D
\\
&
\ge
  \frac{1}{\bar c_D}(1-e^{-\bar c_D t}) P_x(\tau_D>t) \underline{ g}^-_D.
\end{split}
\end{equation*}
Using again \eqref{062904-20}  to estimate $P_x(\tau_D>t)$ from below
and maximizing over $t>0$ we conclude that   
\begin{align}
\label{052904-20a}
\bar u_{ D_S }-u(x)\ge \frac{
  \varphi_D(x)}{\|\varphi_D\|_\infty \bar c_D}\cdot\frac{\underline{ g}^-_D\bar
  c_D /\la_D}{(1+\bar c_D/\la_D)^{1 +\la_D/\bar c_D}} 
\ge \frac{
  \varphi_D(x)}{e\|\varphi_D\|_\infty }\cdot\frac{\underline{ g}^-_D}{\la_D+\bar
  c_D}.
\end{align}
Estimate \eqref{012804-20} follows easily from \eqref{072904-20} and \eqref{052904-20a}.
\end{proof}

\begin{theorem}
\label{thm010807-20}
Assume  the minorization condition $M(\bar c_D, \psi_D^{\bar c_D},
\nu^{\bar c_D}_D)$ (see Definition \ref{minorization}) for some non-negative function $\psi_D^{\bar
  c_D}$ and Borel measure $\nu^{\bar c_D}_D$.   Suppose
furthermore that   $u\in\mathcal U^+_c(g)$.
Then,
\begin{equation}
\label{022804-20ax}
\bar u_{D_S }-u(x)\ge \psi^{\bar c_D}_D(x) \left\{\bar u_{D_S }\int_Dc
v_{c,D}d\nu^{\bar c_D}_D +\int_{D}g^-\,d\nu^{\bar c_D}_D\right\},\quad x\in D.
\end{equation}
\end{theorem}
\begin{proof}
We start with the estimate  \eqref{052904-20b}. Letting $t\to+\infty$
we conclude, upon an application the Fatou lemma, that
\begin{equation*}
%\label{052904-20bb}
\bar u_{ D_S }-u(x)\ge w_{c,D}(x) \bar u_{ D_S }-\mathbb
E_x\left[\int_0^{\tau_D}e_c(s)g(X_s)\,ds\right],\quad t>0.
\end{equation*}
To estimate the first
term in the right hand side we use   (\ref{eq.wcd1}) and  $M(\bar c_D, \psi_D^{\bar c_D},
\nu^{\bar c_D}_D)$. Then, for any $x\in D$, 
\begin{equation*}
%\label{011607-20}
w_{c,D}(x) \bar u_{ D_S }
=\E_x\left[\int_0^{\tau_D}v_{c,D}(X_t)c(X_t)dt\right]\bar u_{ D_S }
\ge \bar u_{ D_S } \psi_D^{\bar c_D}(x)\int_{D}cv_{c,D}d\nu^{\bar c_D}_D
\end{equation*}
and
\begin{align*}
%\label{102904-20}
\mathbb
-\mathbb{E}_x\left[\int_0^{\tau_D}e_c(s)g(X_s)\,ds\right]\ge-\mathbb
  E_x\left[\int_0^{\tau_D}e^{-\bar c_D s}g(X_s) ds\right]
\ge
-  \psi_D^{\bar c_D}(x) \int_{D}g\,d\nu^{\bar c_D}_D.
\end{align*}
Estimate \eqref{022804-20ax} thus follows.
\end{proof}

 {Theorem \ref{thm010807-20} combined with   the
  results implying the minorization condition, obtained in the foregoing, 
allow us to formulate various  quantitative Hopf lemmas. We formulate two such results, which may be of special interests 
in the theory of P.D.E-s}

\begin{theorem}
\label{ultra}
1) Assume that  IR$_m$) holds and $mR_\alpha^D\ll m$ for some $\alpha\ge 0$.  
Then, there exist a
function $\psi_{D,A}:D\to(0,+\infty)$ and a Borel measure $\nu_{D,A}$ on
$D$ such that 
$m\ll \nu_{D,A}$ and any $u\in\mathcal U^+_c(g)$ satisfies
\begin{equation}
\label{022804-20}
\bar u_{D_S }-u(x)\ge  \psi_{D,A}(x)\left\{ \bar u_{D_S }\int_Dc
v_{c,D}d\nu_{D,A}+ \int_{D}g^-\,d\nu_{D,A}\right\},\quad x\in D.
\end{equation}
Moreover, if coefficient $c(\cdot)$ satisfies A3'), then $\int_Dc
v_{c,D}d\nu_{D,A}>0$.

2) Assume that $(P_t)$ is intrinsically ultracontractive. Then there exists $a>0$ such that
\begin{equation}
\label{022804-20apxp}
\bar u_{D_S }-u(x)\ge  a\varphi_D(x)\left\{ \bar u_{D_S }\int_Dc\hat\varphi_D
v_{c,D}dm+ \int_{D}\hat\varphi_Dg^-\,dm\right\},\quad x\in D
\end{equation}
for any $u\in \mathcal U_c(g)$. Moreover, if coefficient $c(\cdot)$ satisfies A3'), then $\int_Dc
v_{c,D}dm>0$.
\end{theorem}
\begin{proof}
Estimate \eqref{022804-20} follows from Theorems \ref{thm010807-20}
and    \ref{th3.hl.nuc1a}, while
\eqref{022804-20apxp} follows from  Theorems \ref{thm010807-20} and   \ref{prop5.1}.
From condition ET)  we conclude that $v_{c,D}(x)>0$ for each $x\in
D$. Therefore, thanks to assumption  A3'), we obtain that  $\int_D c v_{c,D}\,dm >0$.
From this and the fact that $m\ll \nu_{D,A}$, we further conclude that $\int_Dc v_{c,D}d\nu_{D,A}>0$.
\end{proof}

\begin{corollary}
\label{cor010907-20apx}
Assume that $(P_t)$ is intrinsically ultracontractive.  Suppose that   \eqref{sal.sal1} holds and  $c$ satisfies A3'). Then, there exists a constant $a>0$
such that for any   $u(\cdot)$ - non-constant $\ell_d$-a.e.  weak
subsolution to (\ref{eq3.1a}) - and $\hat x\in   \partial D$  with    $u(\hat
x)=\bar u_{D_S }\ge0$  we have 
\begin{equation*}
%\label{012804-20zapx45}
u(\hat x)-u(x)\ge a \phi(\delta_D(x))\left(u(\hat x)+\int_D\hat\varphi_D(y)g^-(y)\,dy\right). 
\end{equation*}
\end{corollary}

\begin{remark}
Theorem \ref{ultra} is a far reaching generalization of the Morel-Oswald 
formulation of the Hopf-Lemma (see \cite{BC}) as can be seen from Corollary \ref{cor010907-20apx}.
The result of ibid. states that: if  $D$ is a smooth bounded domain, then
there exists $c>0$ such that for any $u\in W^{2,2}(D)\cap W^{1,2}_0(D)$ satisfying
\[
-\Delta u=f\quad\mbox{in}\,\, D,
\]
with $f\in L^\infty(D)$, such that $f
\ge0$  we have
\[
u(x)\ge c \delta_D(x)\int_D\delta_D(y)f(y)\,dy,\quad x\in D.
\]
\end{remark}

%\subsubsection{Non-negative eigenfunction and eigenvalue}
%Using the transition semigroup we can define  a non-negative
%eigenvalue and  eigenfunction for the operator   $A$. 

%Below we formulate a number of results concerning the
%lower bounds on $\bar u_{D_S }- u(x)$.  Their common
%denominator    is  the fact that the constant
%$a$ appearing in either estimate \eqref{010104-20av1}, or \eqref{010104-20av1a}
%equals $ \bar u_{D_S }$.  Estimates of this type shall be
% referred to as {\em quantitative Hopf lemmas}.

%\begin{definition}[Quantitative Hopf lemma]
%\label{dfQHL}

%We say that
%that  a quantitative Hopf lemma holds for subsolutions of (\ref{eq3.1a}), if there
%exists  a  function $\psi:D\to(0,+\infty)$ such that  for any $u\in
%\mathcal U_c(g)$ we have
%\begin{equation}
%\label{qHL-1}
%\bar u_{D_S }-u(x)\ge \bar u_{D_S }\psi(x),\quad x\in D
%\end{equation}
%\end{definition} 
%Note that, for obvious reasons,  positivity of  $ \bar u_{D_S }$
 % need not be   assumed in \eqref{qHL-1}. 

 {
\subsection{Properties  of $w_{c,D}$}
\label{sub.sub12}

Here  we take a closer look at the function $w_{c,D}$
 appearing in Proposition \ref{prop010705-20}. 
Recall that for its definition, cf \eqref{wcD}, we require only conditions $A1) -A4)$. 
Obviously $0\le w_{c,D}(x)\le 1,\, x\in D$. Since $w_{c,D}$ is bounded, a simple calculation shows that
\begin{equation}
\label{eq.wcd1}
w_{c,D}(x)=R^D(c v_{c,D})(x)=R^D(c-c w_{c,D} )(x),\quad x\in D.
\end{equation}
Indeed, let $C_t:=\int_0^tc(X_r)\,dr$. Clearly, $v_{c,D}(x)=\mathbb E_xe^{-C_{\tau_D}},\, x\in\BR^d$.
Then,
\begin{equation}
\label{eq.comb9240}
1-v_{c,D}(x)=1-\mathbb E_xe^{-C_{\tau_D}}=\mathbb E_x\left[\int_0^{\tau_D}e^{-C_t}\,dC_t\right]=\mathbb E_x\left[\int_0^{\tau_D}c(X_t)e^{-C_t}\,dt\right].
\end{equation}
On the other hand, by the Markov property,
\[
\mathbf{1}_{\{t<\tau_D\}}v_{c,D}(X_t)=\mathbf{1}_{\{t<\tau_D\}}\mathbb E_x\Big[e^{-\int_t^{\tau_D}c(X_r)\,dr}\big|\FF_t\Big].
\]
Thus,
\begin{align*}
R^D(cv_{c,D})(x)&=\mathbb E_x\left[\int_0^{\tau_D}c(X_t)v_{c,D}(X_t)\,dt\right]=\mathbb E_x\left[\int_0^{\tau_D}c(X_t)e^{-\int_t^{\tau_D}c(X_r)\,dr}\,dt\right]
\\&
=\mathbb E_x\left[\int_0^{\tau_D}c(X_t)e^{-\int_0^tc(X_r)\,dr}\,dt\right].
\end{align*}
This combined with \eqref{eq.comb9240} yields \eqref{eq.wcd1}.

From the definition of $w_{c,D}$ we conclude that
\[
w_{c,D}(x)>0\,\quad\mbox{iff}\quad \,P_x\Big(\int_0^{\tau_D}c(X_r)\,dr>0\Big)=1.
\]

Thanks to \eqref{eq.wcd1} we conclude the following.
\begin{proposition}
\label{prop011009-20} Suppose that
$c$ satisfies assumption A3'). Then,
condition IR$_m$) implies that $w_{c,D}(x)>0$ for all $ x\in D$.
\end{proposition}
\begin{remark}
In particular, the assumption of uniform ellipticity of the local part
of $A$ implies condition IR$_m$), see \cite[Section
7.7.1]{kk01}.   Then, $w_{c,D}$ is
  strictly positive. It turns out that it is also  continuous on $D$,
see Lemma 7.6 of ibid.
\end{remark}

}

%\textcolor{red}{\bf DOTAD}

\appendix

 {

\section{Supplement: Some additional comments on the hypotheses and applications of
  the results}}

\label{sec6zz}

\subsection{Remarks on the existence of  a strong Markovian  solution of the martingale problem}
\label{sec.a4}

Fairly easy to verify conditions on the coefficients of the operator $A$
implying A4) 
appear in the literature, see e.g.  \cite{Stroock,LM},
Chapter 4 of \cite{ethier-kurtz}, \cite{kuhn,kuhn2}, \cite{AP,Hoh,JS}.
Below, we review some of the existing results.  First note that  without any loss of
generality we may and shall assume that
\begin{equation}
\label{eq3.mc1}
N(x,B^c(0,1))=0,\quad x\in\BR^d.
\end{equation}
Indeed, let   $(P_x)_{x\in\bbR^d}$ be a strong Markov process solving
the martingale problem for an operator $A_0$ given by  \eqref{A}, with $N$
replaced by
\[
\tilde N(x,dy)=\mathbf{1}_{B(0,1)}(y)N(x,dy).
\]
By \cite[Proposition 10.2, Section 4]{ethier-kurtz} there exists then a
strong Markov solution to the
problem associated with  $A=A_0+A_1$, where
\[
A_1f(x):=\int_{B^c(0,1)}(f(x+y)-f(x))\,N(x,dy),\quad x\in\bbR^d.
\] 
%Clearly, $A=\tilde A+K$. 
From condition (\ref{eq3.mc1}) and assumption \eqref{MA} we infer 
% \[
% \|b\|_\infty+\|Q\|_\infty+\|\int_{\BR^d\setminus\{0\}}\min\{|y|^2,1\}\, N(\cdot,dy)\|_\infty<\infty,
% \]
\begin{equation}
\label{pxxi}
\lim_{R\rightarrow \infty}\sup_{|x|\le R}\sup_{|\xi|\le 1/R}|p(x,\xi)|=0,
\end{equation}
where $p(x,\xi)$ is the Fourier symbol associated with the operator
$A$, defined as
\begin{equation}
\label{eq3.ex3td}
p(x,\xi):=\frac{1}{2}\sum_{k,\ell=1}^dq_{k,\ell}(x)\xi_k\xi_\ell-i
\sum_{k=1}^db_{k}(x)\xi_k+\int_{\BR^d}\left(1-e^{i\xi\cdot
    y}+\frac{i y\cdot \xi}{1+|y|^2}\right)N(x,dy)
\end{equation}
 for any $x,\xi\in\bbR^d$.
Obviously ${\rm Re}\, p(x,\xi)\ge 0$.

By \cite[Theorem 4.1]{kuhn} condition \eqref{pxxi} implies the existence of a strong
Markovian solution associated with the operator $A$, provided that we
can prove that
there exists a solution to the martingale problem for $A$
and an arbitrary initial Borel probability distribution $\mu$. %any probability measure $\mu$ on $\BR^d$

To state the result concerning the existence of
 a strong Markovian  solution   we formulate  some additional hypotheses.

\begin{itemize}
\item[SE)] 
 In addition to the assumptions made in A1)  suppose that
  the matrix 
${\bf
  Q}(x)$  is {\em uniformly positive
definite} on compact sets, i.e. for any compact set $K\subset \R^d$
there exists $\la_K>0$ such that 
 \begin{equation}
\label{la-K}
\la_K|\xi|^2\le \sum_{i,j=1}^dq_{i,j}(x)\xi_i\xi_j,\quad x\in K,\, \xi=(\xi_1,\ldots,\xi_d)\in\BR^d.
\end{equation}

\item[C)] the mapping $\BR^d\ni x\mapsto {\bf
    Q}(x)\in\bbR^{d\times d}$ is continuous together with
\[
\BR^d\ni x\mapsto N_B(x):=\int_{B}\min\{|y|^2, 1\} \,N(x,dy)
\]
for any Borel  $B\subset B(0,1)$.
\end{itemize}
The following summarizes a few of sufficient conditions for the
validity of A4). 
\begin{theorem}
\label{thm2.3}
Hypothesis A4) is satisfied, provided that   one of the following
conditions are fulfilled:
\begin{itemize}
\item[a)]  \eqref{MA} and \eqref{la-K} are in force, or
\item[b)]  \eqref{MA}, C) hold and ${\bf Q}(x)$ is invertible for every $x\in\BR^d$, or
\item[c)] \eqref{MA}, C) are satisfied and the mapping $x\mapsto b(x)$ is continuous.
\end{itemize}
\end{theorem}
\proof
The fact that a) implies A4) follows from \cite{AP}, see also
\cite{LM}. The result concerning condition b) is a  consequence of
\cite[Theorem III.2.34, p. 159]{JS}, see also \cite{Stroock}. The
implication for
c) follows from
\cite{Hoh}, see also \cite[Theorem 3.24]{BSW}, respectively.
\qed

\subsection{Finiteness of the exit time}

\label{rm3.14} 

Clearly, for the validity of the weak maximum principle (WMP), see
\eqref{032804-20}, we need,  
some additional condition besides  A1) - A4),  to exclude  at  least  the case  $A=0$, in
which the principle obviously fails. 
Recall that in the case of the Laplace operator we have the so called infinite propagation speed of disturbances.
Heuristically,  the necessary condition for the validity of the WMP is the communication (via
the canonical process $(X_t)_{t\ge0}$) between the interior and the exterior of the domain $D$, or in other words propagation 
of the disturbance in the entire $\BR^d$. Observe that if $P_x(\tau_D=+\infty)=1$
for all $ x\in D$ (no communication), then $u=1_{D}$ is a weak subsolution 
of  $A v=0$ (here $c\equiv 0$). Obviously, (\ref{032804-20}) does not hold for $u$. 
So, condition ET), used in Proposition \ref{prop012904-20},  seems  to
be fairly close to optimal for the
validity of
the weak maximum principle. 
Many sufficient conditions can be found in the
literature implying  ET).
Below we review a few of them.

As we have already mentioned in Remark \ref{rmk010211-20} the uniform
ellipticity   on compact sets, see \eqref{la-K},
suffices   for the validity of ET). It is also possible to formulate
sufficient conditions   without assuming  SE).  
 % To formulate it
% let us denote
% \begin{equation}
% \label{eq3.ex3td}
% p(x,\xi):=\frac{1}{2}\sum_{k,\ell=1}^dq_{k,\ell}(x)\xi_k\xi_\ell-i
% \sum_{k=1}^db_{k}(x)\xi_k+\int_{\BR^d}\left(1-e^{i\xi\cdot
%     y}+\frac{i y\cdot \xi}{1+|y|^2}\right)N(x,dy)
% \end{equation}
%  for any $x,\xi\in\bbR^d$.
% Obviously ${\rm Re}\, p(x,\xi)\ge 0$.
Using \eqref{eq3.ex3td} the operator $A$   can be  written as
$$
Au(x)=-\frac{1}{(2\pi)^d}\int_{\bbR^d}e^{ix\cdot \xi}p(x,\xi)\hat u(\xi)d\xi
$$
for any $u$ belonging to $ C^\infty_c(\bbR^d)$ - the set of  $C^\infty$ smooth and compactly
supported functions. 
\begin{enumerate}
\item[(i)]
By \cite[p. 3275]{schilling}
there exist constants $C,c>0$, depending only on $d$,  such that
\begin{equation*}
%\label{schil}
\E_x\left[\int_0^{\infty}1_{B(0,r)}(X_t)dt\right]\le Cr^d \int_{[|\xi|\le c/r]}\frac{d\xi}{\inf_{z\in\bbR^d}{\rm Re}\,p(z,\xi)},\quad x\in D.
\end{equation*}
for all $r>0$. This in particular shows that 
ET)
holds, provided that 
 \begin{equation*}
%\label{022505-20}
\int_{[|\xi|\le r]}\frac{d\xi}{\inf_{z\in\bbR^d}{\rm Re}\,p(z,\xi)}<+\infty\quad\mbox{for some $r>0$.}
\end{equation*}

\item[(ii)] Let $r={\rm diam}\, D$. Suppose that for every $x\in {\rm cl}\,D$ there exists $k(x)\ge 1$ such that
\begin{equation}
\label{eq3.ex1td}
2r|\xi||{\rm Im}\, p(z,\xi)| \le {\rm Re}\, p(z,\xi)\quad\mbox{for } |\xi|\le \frac{1}{k(x)r},\, z\in B(x,r).
\end{equation}
Then, by \cite[Theorem 5.5]{BSW}, there exists $c>0$ such that
\begin{equation}
\label{eq3.ex2td}
P_x(\tau_D>t)\le \frac{c}{t}\left(\sup_{|\xi|\le 1/(rk(x))}\inf_{z\in B(x,r)} {\rm Re}\, p(z,\xi)\right)^{-1}.
\end{equation}
So, if
$$
\sup_{|\xi|\le 1/(rk(x))}\inf_{z\in B(x,r)} {\rm Re}\,
p(z,\xi)<+\infty,\quad x\in\bbR^d,
$$
 then letting $t\rightarrow \infty$, we conclude ET), in
 fact $\E_x\tau_D^{\rho}<+\infty$ for any $\rho\in(0,1)$.
For example consider the operators
\[
A_1u(x)=\frac12\sum_{j,\ell=1}^D q_{j,\ell}(x)u_{x_jx_\ell}(x),\quad A_2 u= \Delta^{s (x)} u(x),
\]
\[
A_3u=-\sum_{j=1}^d|\partial_{x_j}^2|^{s_j}u(x),\quad u\in
W^{2,p}(D)\cap C_b(\bbR^d),
\]
with $p>d$,  $s:\BR^d\rightarrow (0,1)$, and $s_1,\dots,s_d\in (0,1)$ fixed. Symbol $p_1$ for $A_1$ is given by (\ref{eq3.ex3td}) with $b=0, N\equiv 0$, and
symbols $p_2$, $p_3$, for $A_2$ and $A_3$ respectively, are given by 
\[
p_2(x,\xi)=|\xi|^{2s (x)},\quad p_3(\xi)=\sum_{j=1}^d|\xi_j|^{2s_j}.
\]
In all the  cases listed above the imaginary part of the symbols vanishes, so (\ref{eq3.ex1td}) trivially holds.
Now, we see that  by (\ref{eq3.ex2td}), ET) holds for $p_2$ and $p_3$,  and if for some $j,\ell\in\{1,\dots,d\}$, $q_{j,\ell}$ is strictly positive on compacts, then
 ET) holds for $p_1$ as well.
\item[(iii)]  Let $r={\rm diam}\, D$. Then, by \cite[Proposition
  3.7]{SW1}  ET) holds, provided that for every $x\in D$ we have
\[
\inf_{z\in B(x,r)}\int_{|y|\ge 3r}N(z,dy)>0.
\]
\end{enumerate}

\subsection{Existence of the principal eigenvalue and eigenfunction}
\label{sub.sub.uewo4.2}
{
Frequently, in applications we have  extra information about the
structure of semigroup $(P^D_t)_{t\ge 0}$ in particular consider the
case when
\[
P^D_tf(x)=\int_Dp_D(t,x,y)f(y)\,m(dy),\quad f\in B^+_b(D),\,t>0
\]
for some $p_D:(0,\infty)\times D\times D\to(0,\infty)$ and $m$ a
finite Borel measure on $D$. In this case 
some assumption about integrability of $p_D$, sufficing for compactness of $(P^D_t)_{t\ge 0}$,  implies Eig).
For example if
\[
\int_D\int_D p_D^2(t,x,y)\,m(dx)m( dy)<\infty,\quad t>0,
\]
then $P^D_t$ is a Hilbert-Schmidt operator on $L^2(D;m)$  for each $t>0$, hence it is compact.
By Jentzsch's theorem (see \cite[Theorem V.6.6, p. 337]{schaefer}) condition Eig) holds. 
If we know that for some  $1<p<q$ the hypercontractivity condition
holds, i.e. for any $t>0$ there exists $c_t>0$
such that
\[
\|P^D_t f\|_{L^q(D;m)}\le c_t\|f\|_{L^p(D;m)},\quad f\in L^p(D;m),
\]
then $P^D_t$ is compact on $L^p(D;m)$  for each $t>0$
and, again  by Jentzsch's theorem, Eig) holds. Moreover $\la_D>0$ is a
unique simple eigenvalue of $(P_t^D)_{t\ge0}$. The corresponding
eigenfuction $\varphi_D$ can be chosen to be strictly
positive. Normalizing it by letting $\int_D \varphi_Ddm=1$ we can uniquely
determine its choice.   The proof of
compactness is analogous to the argument presented  in \cite[Section 7.1.1]{kk01}.}

{
In some cases however the aforementioned properties of  the semigroup are not
so easy to come by.
Then one can try to apply the Jentzsch theorem having sufficient information about 
the resolvent operator $R^D_\alpha$.}
In this way one can e.g. conclude Eig) from  hypotheses A2) and SE),
see \cite[Theorem 5.1]{kk01}.

\bigskip

\subsection{Viscosity subsolutions}

\label{sec6.5}

Suppose   that the transition probability semigroup $(P_t)_{t\ge0}$  associated
with the martingale problem is strongly Feller,
i.e. $P_t(B_b(\BR^d))\subset C_b(\BR^d)$, $t>0$.
Assume that $D$ is Dirichlet regular, i.e. $P_x(\tau_D>0)=0,\, x\in\partial D$.
Furthermore, suppose that $c,g, q_{i,j},b_i\in C_b(\BR^d)$,
$i,j=1,\ldots,d$, the Levy kernel $N$ satisfies condition C) (see Section \ref{sec.a4}) and
the family ${\cal O}_R$ of open Dirichlet regular subsets of $\BR^d$ forms a base
for the Euclidean metric.

Let $u\in C_b(\bbR^d)$ be a viscosity subsolution  to
  \eqref{eq3.1a}.  {Using a suitable comparison principle for viscosity
subsolutions one could prove that, under the assumptions made in
    the foregoing, $u$  is  also a weak subsolution to
  \eqref{eq3.1a}.
This can be seen as follows.} With no loss of generality we may assume
that   
\begin{equation}
\label{011601-21}
u(x)\ge 0, \quad x\in D^c:=\BR^d\setminus D.
\end{equation}
  Indeed, otherwise
we would consider $\tilde u(x):=u(x)+h(x)$, where 
$h(x):=\bar u_{\bbR^d}\mathbb E_x\left[e_c(\tau_D) \right]$. The latter
  is a weak solution to \eqref{eq3.1ab} satisfying $h(x)=\bar
  u_{\bbR^d}$, $x\in D^c$, due the fact that $D\in {\cal O}_R$. 
By virtue of e.g. \cite[Theorem 2.2]{Pardoux} it is also a viscosity
solution. 
  Therefore $\tilde u(x)$ is  a viscosity subsolution  to
  \eqref{eq3.1a} that satisfies \eqref{011601-21}. The fact that $u$
  is also a weak subsolution would follow, if we can prove that $\tilde
  u$ has this property.

Let $V\subset D$ and $V\in\mathcal O_R$. By the assumptions made in the foregoing  
\begin{equation}
\label{011201-21}
w_V(x):=\mathbb E_x\left[e_c(\tau_V)u(X_{\tau_V})\right]+\mathbb E_x\left[ \int_0^{\tau_V}e_c(r)g(X_r)\,dr\right],\quad x\in \BR^d
\end{equation}
is continuous on $\BR^d$.
Thanks to the aforementioned result of  \cite{Pardoux}  it
is a viscosity solution to \eqref{eq3.1a} with $D$ replaced by $V$.
Since $V$ is regular, it satisfies the exterior  condition
$w(y)=u(y),\, y\in V^c$. 
Under appropriate assumptions, see e.g. \cite[Theorem 1.2]{BCI}, we can use a comparison principle for viscosity
solutions to the exterior Dirichlet problem for  \eqref{eq3.1a} with $D$ replaced by $V$.
It allows us to  conclude    that
\begin{equation}
\label{eq.str019}
u(x)\le w_V(x),\quad   x\in \BR^d.
\end{equation}
The above inequality holds for arbitrary $V\subset D$, with $V\in\mathcal O_R$.
To conclude that $u$ is a weak subsolution,   cf \eqref{subsol.g.eq}, we
need to  replace the family of stopping times  $(\tau_V)_{V\in\mathcal O_R}$  by the family $(\tau_D\wedge t)_{t\ge 0}$.
%We shall prove that this is indeed the case  for any $D$  Dirichlet regular.
Let 
$$
v:=u-R^{c,D}g.
$$
 From \eqref{011201-21} and \eqref{eq.str019}
we can write
\begin{equation}
\label{011201-21z}
u(x)-R^{c,V}g(x) \le w_V(x)-R^{c,V}g(x)=\mathbb E_x^c u(X_{\tau_V}),\quad x\in \BR^d.
\end{equation}
Since 
$$
R^{c,D}g(x)=R^{c,V}g(x)+ \bbE_x^cR^{c,D}g(X_{\tau_V}),
$$
we conclude from \eqref{011201-21z} that
\begin{equation}
\label{011702-21}
\mathbb E^{c}_x   v(X_{\tau_V})\ge v(x),\quad x\in V.
\end{equation}
Let $\tilde v:=\overline{ v^+}_D-v$.
Obviously $\tilde v\ge 0$  in $D$. From \eqref{011702-21} we have
$$
  \tilde v(x)=\overline{ v^+}_D-v(x)\ge \mathbb E^{c}_x \tilde v(X_{\tau_V})\ge \mathbb E^{c,D}_x \tilde v(X_{\tau_V})
  , \quad x\in V.
$$
Since  the  sets $V$ belonging to ${\cal O}_R$ and contained
in $D$ form a base for the Euclidean metric in $D$, it follows from \cite[Theorem II.5.1]{bg} 
that $\tilde v$ is an excessive function with respect to
$(P^{c,D}_t)$. Thus, $\big(\tilde v(X_{t\wedge \tau_D})\big)_{t\ge0}$ is a
supermartingale  with respect to $P^{c,D}_x$. The above implies in turn that $\big( v(X_{t\wedge \tau_D})\big)_{t\ge0}$ is a
submartingale  with respect to this measure.
 %  By \eqref{eq.str019} we have
%  $u(X_{\tau_D-})\le u(X_{\tau_D}),\, P^c_x$-a.s. 
%   \textcolor{red}{Consequently, we conclude that $Y_t:=
% u(X_{t\wedge \tau_D})$, $t\ge0$
% is a submartingale with respect to $P^c_x$???}
Since $v(X_{\tau_D})=u(X_{\tau_D})$ is non-negative, by the optional sampling
theorem, we have
\begin{equation*}
\begin{split}
&v(x)\le \mathbb E_x^c\left[v(X_{t}),\,t<\tau_D\right]\le \mathbb E_x^c\left[v(X_{t}),\,t<\tau_D\right]+ \mathbb E_x^c\left[v(X_{\tau_D}),\,t\ge\tau_D\right]\\
&
=\mathbb E_x^c v(X_{t\wedge\tau_D}), \quad x\in D
\end{split}
\end{equation*}
and \eqref{subsol.g.eq} follows.
%\qed
% $\mathbb E^c_xY_0\le \mathbb E^c_xY_{  t}$ and
% this implies 

\subsection*{Acknowledgements}
{\small  T. Klimsiak is supported by Polish National Science Centre:
  Grant No. 2017/25/B/ST1/00878.  Both T. Klimsiak and T. Komorowski  acknowledge the
support of the  Polish  National Science Centre:
Grant No.  2020/37/B/ST1/00426.}

\end{document}